\documentclass[11pt]{amsart}

 \usepackage{amsfonts,graphics,amsmath,amsthm,amsfonts,amscd,amssymb,amsmath,latexsym,multicol,
 mathrsfs}
\usepackage{epsfig,url}
\usepackage{flafter}
\usepackage{fancyhdr}
\usepackage{hyperref}
\hypersetup{colorlinks=true, linkcolor=black}

 \usepackage[matrix, arrow]{xy}

\DeclareMathOperator{\Bs}{Bs}

\DeclareMathOperator{\Supp}{Supp}

 \numberwithin{equation}{subsection}
 \numberwithin{footnote}{subsection}

 \newtheorem{cor}[subsection]{Corollary}
 \newtheorem{lem}[subsection]{Lemma}
 
 \newtheorem{thm}[subsection]{Theorem}

 \newtheorem{quest}[subsection]{Question}

{
\theoremstyle{upright}

 \newtheorem{rem}[subsection]{Remark}

}

 \newcommand{\N}{\mathbb N}
 \newcommand{\PP}{\mathbb P}
 
 \newcommand{\Q}{\mathbb Q}
 \newcommand{\R}{\mathbb R}
 
  \newcommand{\C}{\mathbb C}
 \newcommand{\bir}{\dashrightarrow}
 \newcommand{\rddown}[1]{\left\lfloor{#1}\right\rfloor} 
 
  \newcommand{\Nklt}{{\rm{Nklt}}}

\title{\large O\MakeLowercase{n connectedness of non-klt loci of singularities of pairs}}
\author{\large C\MakeLowercase{aucher} B\MakeLowercase{irkar}}
\date{\today}
\begin{document}
\maketitle

\begin{abstract}
We study the non-klt locus of singularities of pairs. We show that given a pair $(X,B)$ and a projective 
morphism $X\to Z$ with connected fibres such that  
$-(K_X+B)$ is nef over $Z$, the non-klt locus of $(X,B)$ has at most two connected components near 
each fibre of $X\to Z$. This was conjectured by Hacon and Han. 

In a different direction we answer a question of Mark Gross on connectedness of the non-klt loci 
of certain pairs. This is motivated by constructions in Mirror Symmetry.

\end{abstract}



\section{\bf Introduction}

We work over an algebraically closed field of characteristic zero.

Pairs and their singularities play a fundamental role in higher dimensional algebraic geometry. 
Let's consider the simplest kind of pair, that is, a projective 
log smooth pair $(X,B)$ where $X$ is a smooth projective variety and 
$B=\sum b_i B_i$ is a divisor with simple normal crossing singularities and real coefficients $b_i\ge 0$; 
here $B_i$ are distinct prime divisors. 
In this setting, the \emph{non-klt locus} $\Nklt(X,B)$ of $(X,B)$ is the union of the $B_i$ with $b_i\ge 1$. 
In general, $\Nklt(X,B)$ can have any number of connected components as a topological space with 
the Zariski topology. But in special situations the non-klt locus exhibits interesting behaviour. 
For example, Shokurov [\ref{Shokurov-log-flips}, Connectedness Lemma 5.7] proved in the early 1990's that if $X$ is a surface and 
$-(K_X+B)$ is ample, then $\Nklt(X,B)$ is connected. 
This was generalised to higher dimensions by Koll\'ar [\ref{Kollar-flip-abundance}, Theorem 17.4], 
which is known as the connectedness lemma or connectedness principle. The same holds if we only assume 
$-(K_X+B)$ to be nef and big, and it also holds in the relative setting when 
$X$ is defined over a base $Z$ in which case connectedness holds near each fibre of $X\to Z$ assuming 
$X\to Z$ has connected fibres.

The connectedness principle plays an important role in higher dimensional algebraic geometry. For example, 
it is used in the proof of existence of flips [\ref{Shokurov-log-flips}][\ref{Shokurov-pl-flips}], in
the proof of inversion of adjunction [\ref{Kollar-flip-abundance}, Theorem 17.6], 
in the proofs of boundedness of complements and boundedness of Fano 
varieties [\ref{B-compl}, Proposition 5.1, Proposition 6.7], in the proof of the Jordan property of birational 
automorphism groups of rationally connected varieties 
[\ref{Prokhorov-Shramov}], in birational rigidity of Fano varieties 
[\ref{Pukhlikov}], in the study of adjunction for fibre spaces [\ref{Shokurov-log-adjunction}], etc.
 
A natural question is what happens if we only assume that $-(K_X+B)$ is nef? In this case, $\Nklt(X,B)$ 
may not be connected. 
The easiest example is to take $X=\PP^1$, $B=B_1+B_2$ where $B_i$ are distinct points. 
It turns out that when connectedness fails this simple example is in a sense the reason.
When $K_X+B\equiv 0$ and the coefficients of $B$ are $\le 1$, Shokurov [\ref{Shokurov-log-flips}, Theorem 6.9] 
in dimension two and Koll\'ar and Kov\'acs [\ref{Kollar-Kovasc}][\ref{Kollar-singMMP}, Theorem 4.40] in higher dimension 
showed that if connectedness fails, then   
$X$ is birational to a (possibly singular) model $(X',B')$ admitting a contraction $X'\to Y'$ where the general fibres 
are $\PP^1$ and $\rddown{B'}$ has exactly two disjoint components both horizontal over $Y'$ 
($X'$ is obtained by running a minimal model program on $-\rddown{B}$). 

Hacon and Han [\ref{Hacon-Han}] investigated the above phenomenon more generally. They showed that if $\dim X\le 4$ and 
if $-(K_X+B)$ is nef, then $\Nklt(X,B)$ has at most two connected components. They conjectured that 
this holds in every dimension and then showed that it follows from the termination of klt flips conjecture.
One of our main results is to verify their conjecture without assuming termination of klt flips 
(see Theorem \ref{t-anti-nef-pairs-1}). 

Another main result of this paper (see Theorem \ref{t-mirror-symmetry} and Corollaries \ref{cor-mirror-symmetry-connectedness}, \ref{cor-mirror-symmetry}), 
in a different but somewhat related direction,  
is an answer to the following question of Mark Gross. The question   
is motivated by constructions in Mirror Symmetry. In fact this work started in response to this 
question which was communicated privately.

In [\ref{gross-s-II}], Gross and Siebert construct mirror families to pairs $(X,B)$ where $X$
is a non-singular projective variety and $B$ is a simple normal crossings
divisor satisfying certain hypotheses, see [\ref{gross-s-II}, Assumptions 1.1]. It is not
clear how strong these hypotheses are, and thus it is important to know
which pairs $(X,B)$ might satisfy them. These hypotheses include  
conditions (ii) and (iii) of Question \ref{q-gross} as well as the desired connectedness in the 
question. The conditions (ii) and (iii) are essential for mirror
symmetry: mirrors are not expected to exist without similar assumptions.
On the other hand, the connectedness condition would appear to be a
technical one. Thus Corollary 1.5 below gives a useful criterion for checking
when this connectedness assumption may hold, and in particular when the
mirror to the pair $(X,B)$ may exist. 

\begin{quest}\label{q-gross}
Consider the following setup:
\begin{itemize}
\item[(i)] $(X,B)$ is a projective log smooth pair where $B=\sum B_i$ is reduced,

\item[(ii)] $K_X+B\equiv \sum a_iB_i$ with $a_i\ge 0$ real numbers (we say $B_i$ is good if $a_i=0$),

\item[(iii)]  $(X,C)$ has a zero-dimensional stratum $x$ where $C$ is the sum of the good divisors, 

\item[(iv)] for each stratum $V$ of $(X,C)$, define $K_V+C_V=(K_X+C)|_V$ by adjunction. 
\end{itemize}
Then under what conditions is $C_V$ connected for every stratum $V$ of dimension $\ge 2$? 
\end{quest}
Here by a stratum of $(X,C)$ we mean $X$ itself or an irreducible component of the intersection of 
any subset of the irreducible components of $C$.

Gross pointed out that if $(X,B)$ has a good log minimal model which is log Calabi--Yau, that is, a 
log minimal model $(X',B')$ with $K_{X'}+B'\sim_\Q 0$, then the connectedness in the question  
holds by [\ref{Kollar-singMMP}, Theorem 4.40] mentioned above. 
If $(X,B)$ has Kodaira dimension zero, then standard conjectures of the minimal 
model program, including the abundance conjecture, imply that $(X,B)$ has such a log minimal model 
but the current technology of the minimal model program is not enough to guarantee existence of good 
log minimal models in general in dimension $\ge 4$. 
The problem is then to find other reasonable assumptions, instead of Kodaira dimension zero, 
so that a good log minimal model exists or at least that the desired connectedness holds. 

It turns out that in the setting of the question it is enough to assume the following property 
which is of a numerical nature so probably easy to check in explicit examples. Let 
$\phi\colon Y\to X$ be the blowup of $X$ at the zero-dimensional stratum $x\in X$ and let $E$ be the exceptional divisor. 
Assume that $\phi^*(K_X+B)-tE$ is not pseudo-effective for any real number $t>0$ (a divisor is 
pseudo-effective if it is numerically the limit of effective divisors). 
Then $C_V$ is connected for every stratum $V$ of dimension $\ge 2$. See Corollary 
\ref{cor-mirror-symmetry-connectedness} 
for a statement that works in a much more general setting. 

The non-pseudo-effectivity assumption of the previous paragraph is not as restrictive as it may seem. In fact, it is 
conjecturally equivalent to $(X,B)$ having Kodaira dimension zero. Indeed, assuming standard conjectures of the 
minimal model program, $(X,B)$ has a log minimal model $(X',B')$ where 
$K_{X'}+B'$ is semi-ample, that is, $|m(K_{X'}+B')|$ is base point free for some sufficiently 
divisible $m\in \N$. Thus there is a contraction $g\colon X'\to Z'$ such that 
$K_{X'}+B'\sim_\Q g^*H$ where $H$ is an ample $\Q$-divisor. Now $X\bir X'$ is an isomorphism 
near $x$ because $x$ does not belong to $\Supp \sum a_iB_i$. If $\dim Z'>0$, then we can find 
$0\le P'\sim_\Q K_{X'}+B'$ so that $x'\in \Supp P'$ where $x'$ is the image of $x$. 
This gives $0\le P\sim_\Q K_X+B$ so that $x\in \Supp P$, so in this case 
$$
\phi^*(K_X+B)-tE\sim_\Q \phi^*P-tE
$$ 
is pseudo-effective for some $t>0$, a contradiction. Therefore, $\dim Z'=0$ which exactly means 
that $(X,B)$ has Kodaira dimension zero. Conversely,
assume $(X,B)$ has Kodaira dimension zero. Then $\phi^*(K_X+B)-tE$ is not pseudo-effective 
for any $t>0$ otherwise one can check that $\phi'^*(K_{X'}+B')-tE'$ is pseudo-effective for some $t>0$ where 
$\phi'$ is the blowup of $X'$ at $x'$ and $E'$ is the exceptional divisor, which is not 
possible as $\phi'^*(K_{X'}+B')\sim_\Q 0$.

Since conditions similar to the above non-pseudo-effectivity condition appear again in this text 
we make a definition to ease notation. 
Given a projective variety $X$, a pseudo-effective $\R$-Cartier $\R$-divisor $L$ on $X$, and a prime divisor 
$S$ over $X$ (that is on birational models of $X$) we define the threshold
$$
\tau_S(L)=\sup\{t\in \R^{\ge 0}\mid \phi^*L-tS ~~\mbox{is pseudo-effective}\}
$$
where $\phi\colon W\to X$ is any resolution on which $S$ is a divisor. This is independent of the 
choice of the resolution, see Lemma \ref{l-tau-independent-of-phi}. A relative version 
of the threshold can similarly be defined when $X$ is projective over a base $Z$, in which 
case we denote it by $\tau_S(L/Z)$.
\\

{\textbf{{Non-klt loci of anti-nef pairs.}}}
Our first main result concerns the non-klt loci of pairs $(X,B)$ with $-(K_X+B)$ nef over some base. 

\begin{thm}\label{t-anti-nef-pairs-1}
Let $(X,B)$ be a pair and $f\colon X\to Z$ be a contraction. Assume $-(K_X+B)$ is nef over $Z$. 
Assume that the fibre of $\Nklt(X,B)\to Z$ over some point $z\in Z$ is not connected. Then we have:
\begin{enumerate}
\item $\Nklt(X,B)\to Z$ is surjective and its fibre over $z$ has exactly two connected components,

\item the pair $(X,B)$ is lc over $z$ and after base change to an \'etale neighbourhood of $z$, 
there exist a resolution $\phi\colon X'\to X$ and a contraction $X'\to Y'/Z$ such that if 
$$
K_{X'}+B':=\phi^*(K_X+B)
$$ 
and if $F'$ is a general fibre of $X'\to Y'$, 
then $(F',B'|_{F'})$ is isomorphic to $(\PP^1,p_1+p_2)$ for distinct points $p_1,p_2$.

Moreover, $\rddown{B'}$ has two disjoint irreducible components $S',T'$, both horizontal over $Y'$, and 
the images of $S',T'$ on $X$ are the two connected components of $\Nklt(X,B)$.  
\end{enumerate}
\end{thm}

As mentioned above, the theorem implies [\ref{Hacon-Han}, Conjecture 1.1]. After completion of this work we learnt 
that Filipazzi and Svaldi [\ref{FS}, Theorem 1.1] have also proved this result using different arguments. We also learnt from 
Shokurov that similar results appeared in his unpublished work.

Here is an example of $(X,B)$ as in the theorem on which the non-klt locus has two components, 
one zero-dimensional and the other one-dimensional. Let $X=\PP^2$, $B_1$ be a line, $x$ 
a closed point not contained in $B_1$, and $B_2,\dots,B_5$ be distinct lines passing through 
$x$. Letting 
$$
B=B_1+\frac{1}{2}(B_2+\dots+B_5)
$$ 
we can see that 
$K_X+B\equiv 0$ and $(X,B)$ is lc and $\Nklt(X,B)$ has two components, 
one is $x$ and the other is $B_1$. In this example, $X'=\mathbb{F}_1\to X$ is the blowup at $x$, 
$X'\to Y'=\PP^1$ is the corresponding $\PP^1$-bundle, $S'$ is the birational transform of $B_1$, and 
$T'$ is the exceptional divisor of $X'\to X$.

The theorem was already known (in this or other forms) 
\begin{itemize}
\item in dimension two  [\ref{Shokurov-log-flips}, Theorem 6.9][\ref{Prokhorov}, \S 3.3], 

\item in case $(X,B)$ is dlt and $K_X+B\equiv 0$ 
assuming termination of klt flips [\ref{Fujino}, Proposition 2.1] and without assuming this termination 
[\ref{Kollar-Kovasc}][\ref{Kollar-singMMP}, Proposition 4.37], and 

\item in any dimension assuming termination of klt flips and in dimension $\le 4$ without assuming 
termination [\ref{Hacon-Han}, Theorem 1.2]. 
\end{itemize}

Also if $K_X+B\sim_\R 0$ and if $(X,B)$ is not lc, 
it was proved in [\ref{ambro-injectivity}, Theorem 6.3] that the non-klt locus, even the non-lc locus, 
of $(X,B)$ is connected. 

In the opposite direction we have the following result.

\begin{thm}\label{t-anti-nef-pairs-2}
Let $(X,B)$ be a pair and $f\colon X\to Z$ be a contraction. Assume $-(K_X+B)$ is nef over $Z$. 
Then the fibres of $\Nklt(X,B)\to Z$ are  connected if any of the following conditions holds:  
\begin{enumerate}
\item  $-(K_X+B)$ is big over $Z$, or 

\item $\Nklt(X,B)\to Z$ is not surjective, or 

\item  $\tau_S(-(K_X+B)/Z)>0$ for every non-klt place $S$ of $(X,B)$.
\end{enumerate}
\end{thm}

Here by a non-klt place we mean a prime divisor $S$ over $X$ (that is, on birational models of $X$) 
such that the log discrepancy $a(S,X,B)\le 0$. 
Note that in view of Theorem \ref{t-anti-nef-pairs-1} we can add another case 
in which the theorem holds, that is, the fibre of 
$\Nklt(X,B)\to Z$ over a point $z\in Z$ is connected if $(X,B)$ is not lc over $z$.

Case (1) is essentially the connectedness principle mentioned above. 
Case (3) implies cases (1) and (2) but in practice we first prove cases (1),(2) and then 
derive case (3) from Theorem \ref{t-anti-nef-pairs-1}. There are situations where one can apply (3) 
but not (1) and (2). For example, consider the following: 
assume $-(K_X+B)$ is semi-ample over $Z$ defining a non-birational contraction $X\to T/Z$; 
assume that $\Nklt(X,B)\to T$ is not surjective but $\Nklt(X,B)\to Z$ is surjective; 
then $\tau_S(-(K_X+B)/Z)>0$ for every non-klt place $S$ of $(X,B)$, so we can apply (3).  

See also [\ref{Hacon-Han}, Corollary 1.3] for some special situations in dimension $\le 4$ 
or any dimension assuming termination of klt flips. 

\bigskip

{\textbf{{Non-klt loci for Mirror Symmetry.}}}
The following result is the main step towards answering Question \ref{q-gross} which works in a much more general setting. 

\begin{thm}\label{t-mirror-symmetry}
Assume that 
\begin{enumerate}
\item $(X,B)$ is a projective $\Q$-factorial dlt pair where $B$ is a $\Q$-divisor,

\item $K_X+B$ is pseudo-effective, 

\item $x\in X$ is a zero-dimensional non-klt centre of $(X,B)$,
 
\item $x$ is not contained in the restricted base locus ${\bf B}_-(K_X+B)$, 
 
\item if $\phi\colon Y\to X$ is the blowup at $x$ with exceptional divisor $E$, then 
$\tau_E(K_X+B)=0$, i.e.  $\phi^*(K_X+B)-tE$ is not pseudo-effective for any real number $t>0$.
\end{enumerate}

Then $(X,B)$ has a good log minimal model which is log Calabi--Yau. More precisely, we can run a 
minimal model program on $K_X+B$ ending with a log minimal model $(X',B')$ with $K_{X'}+B'\sim_\Q 0$.
\end{thm}

Note that in the theorem, $(X,B)$ is dlt and $x\in X$ is its zero-dimensional non-klt centre, 
so $x$ is a smooth point of $X$ and $\Supp B$ is simple normal crossing near $x$. 

Recall that for a $\Q$-divisor $L$ on a normal projective variety, the stable base locus is defined as 
$$
{\bf B}(L):=\bigcap_m \Bs|mL|
$$
where $m$ runs over the natural numbers such that $mL$ is an integral divisor. The restricted 
base locus of $L$ is defined as 
$$
{\bf B}_-(L):=\bigcup_{\epsilon\in\Q^{>0}} {\bf B}(L+ \epsilon A)
$$
where $A$ is any fixed ample divisor (the locus is independent of the choice of $A$).

\begin{cor}\label{cor-mirror-symmetry-connectedness}
Under the assumptions of Theorem \ref{t-mirror-symmetry}, suppose that no non-klt centre of 
$(X,B)$ is contained in ${\bf B}_-(K_X+B)$. For $V=X$ or $V$ a non-klt centre of $(X,B)$, 
define $K_V+B_V=(K_X+B)|_V$. Then $\Nklt(V,B_V)$ is connected when $\dim V\ge 2$.

\end{cor}

In the setting of Question \ref{q-gross}, conditions (1)-(4) of the theorem are automatically satisfied: 
indeed, $(X,B)$ is log smooth so we have (1); $K_X+B\equiv \sum a_iB_i$ with $a_i\ge 0$, 
so $K_X+B$ is pseudo-effective which is (2); $x$ is a zero-dimensional stratum of $(X,C)$, 
so it is a non-klt centre of both $(X,C)$ and $(X,B)$, so we have (3); $x$ is contained only in the good components of $B$, 
so $x$ is not contained in $\Supp \sum a_iB_i$, hence $x$ is not contained in 
${\bf B}_-(K_X+B)\subseteq \Supp \sum a_iB_i$ 
so we have (4). 

Applying Theorem \ref{t-mirror-symmetry} and Corollary \ref{cor-mirror-symmetry-connectedness}, 
we will prove the following answer to Question \ref{q-gross}.

\begin{cor}\label{cor-mirror-symmetry}
Assume that 
\begin{enumerate}
\item $(X,B)$ is a projective log smooth pair where $B=\sum B_i$ is reduced,

\item $K_X+B\equiv \sum a_iB_i$ with $a_i\ge 0$ real numbers (we say $B_i$ is good if $a_i=0$),

\item $(X,C)$ has a zero-dimensional stratum $x$ where $C$ is the sum of the good divisors,

\item for each stratum $V$ of $(X,C)$, define $K_V+C_V=(K_X+C)|_V$ by adjunction, 

\item if $\phi\colon Y\to X$ is the blowup at $x$ with exceptional divisor $E$, then 
$\tau_E(K_X+B)=0$. 
\end{enumerate}
Then $(X,B)$ has a good log minimal model which is log Calabi--Yau, 
and $C_V$ is connected for every stratum $V$ of dimension $\ge 2$.
\end{cor}

\bigskip


{\bf Plan of the paper.}
We will prove Theorems \ref{t-anti-nef-pairs-1} and \ref{t-anti-nef-pairs-2} in Section 3 and 
Theorem \ref{t-mirror-symmetry} and Corollaries \ref{cor-mirror-symmetry-connectedness} and 
\ref{cor-mirror-symmetry} in Section 4. 
We will actually prove more general forms of these results in the setting of generalised pairs. 
Generalised pairs play a key role in the proofs.

\bigskip
{\textbf{{Acknowledgements.}}}
This work was done with the support of Cambridge University and the Royal Society. Revision was done with  the support of Tsinghua University. Thanks to Mark Gross for posing Question \ref{q-gross} which triggered the start of this work, for related discussions, and for helping with some remarks in the introduction. Thanks to Yifei Chen, Christopher Hacon, Vyacheslav V. Shokurov,  
and the referees for helpful comments and corrections.


\section{\bf Preliminaries}

All varieties in this paper are quasi-projective over an algebraically closed field of characteristic zero unless otherwise 
stated.

\subsection{Contractions}
A \emph{contraction} is a projective morphism $f\colon X\to Y$ of varieties 
such that $f_*\mathcal{O}_X=\mathcal{O}_Y$ ($f$ is not necessarily birational). In particular, $f$ has connected fibres and 
if $X\to Z\to Y$ is the Stein factorisation of $f$, then $Z\to Y$ is an isomorphism.

\subsection{Pseudo-effective thresholds}
Given a projective morphism $X\to Z$ of varieties, a pseudo-effective $\R$-Cartier $\R$-divisor $L$ on $X$, and a prime divisor 
$S$ over $X$ (that is on birational models of $X$) we define the pseudo-effective threshold $\tau_S(L/Z)$ as follows. Pick a 
birational contraction $\phi\colon Y\to X$ from a normal variety so that $S$ is a $\Q$-Cartier divisor on $Y$, e.g. 
a resolution of singularities of $X$. Then define 
$$
\tau_S(L/Z)=\sup\{t\in \R^{\ge 0}\mid \phi^*L-tS ~~\mbox{is pseudo-effective$/Z$}\}.
$$
The next lemma shows that this is well-defined.
When $Z$ is a point, we will simply denote the threshold by $\tau_S(L)$.

\begin{lem}\label{l-tau-independent-of-phi}
The threshold $\tau_S(L/Z)$ is independent of the choice of $\phi\colon Y\to X$. 
\end{lem}
\begin{proof}
Let $\phi\colon Y\to X$ and $\phi'\colon Y'\to X$ be birational morphisms from normal varieties so that $S,S'$ 
are $\Q$-Cartier divisors on $Y,Y'$, respectively, where $S,S'$ represent the same divisor over $X$, 
that is, $S'$ is the birational transform of $S$. We denote the corresponding thresholds by 
$\tau_S^\phi(L/Z)$ and $\tau_{S'}^{\phi'}(L/Z)$. 
It is enough to show that $\tau_S^\phi(L/Z)=\tau_{S'}^{\phi'}(L/Z)$ when the induced map $\psi\colon Y'\bir Y$ 
is a morphism because in the general case we can use a common resolution of $Y',Y$. 

If $\phi'^*L-tS'$ is pseudo-effective$/Z$, then obviously 
the pushdown 
$$
\psi_*(\phi'^*L-tS')=\phi^*L-tS
$$ 
is pseudo-effective$/Z$. Thus  $\tau_{S'}^{\phi'}(L/Z)\le \tau_S^\phi(L/Z)$. 
Conversely, if  $\phi^*L-tS$ is pseudo-effective$/Z$, then $\psi^*(\phi^*L-tS)$ is pseudo-effective$/Z$. 
But 
$$
\psi^*(\phi^*L-tS)=\phi'^*L-t\psi^*S\le \phi'^*L-tS',
$$
as $\psi^*S\ge S'$, hence $\phi'^*L-tS'$ is pseudo-effective$/Z$ (this is where we use the $\Q$-Cartier condition of $S$ 
so that we can take the pullback $\psi^*S$). Thus  $\tau_{S'}^{\phi'}(L/Z)\ge \tau_S^\phi(L/Z)$ which in turn 
implies the equality $\tau_{S'}^{\phi'}(L/Z)=\tau_S^\phi(L/Z)$.

\end{proof}

\subsection{Pairs}
A \emph{sub-pair} $(X,B)$ consists of a normal quasi-projective variety $X$ and an $\R$-divisor 
$B$ such that $K_X+B$ is $\R$-Cartier. 
If the coefficients of $B$ are at most $1$ we say $B$ is a 
\emph{sub-boundary}, and if in addition $B\ge 0$, 
we say $B$ is a \emph{boundary}. A sub-pair $(X,B)$ is called a \emph{pair} if $B\ge 0$.

Let $\phi\colon W\to X$ be a log resolution of a sub-pair $(X,B)$. Let $K_W+B_W$ be the 
pullback of $K_X+B$. The \emph{log discrepancy} of a prime divisor $D$ on $W$ with respect to $(X,B)$ 
is $1-\mu_DB_W$ and it is denoted by $a(D,X,B)$.
We say $(X,B)$ is \emph{sub-lc} (resp. \emph{sub-klt})(resp. \emph{sub-$\epsilon$-lc}) 
if $a(D,X,B)$ is $\ge 0$ (resp. $>0$)(resp. $\ge \epsilon$) for every $D$. When $(X,B)$ 
is a pair we remove the sub and say the pair is lc, etc. Note that if $(X,B)$ is an lc pair, then 
the coefficients of $B$ necessarily belong to $[0,1]$. 

Let $(X,B)$ be a sub-pair. A \emph{non-klt place} of $(X,B)$ is a prime divisor $D$ on 
birational models of $X$ such that $a(D,X,B)\le 0$. A \emph{non-klt centre} is the image on 
$X$ of a non-klt place. When $(X,B)$ is lc, a non-klt centre is also called an 
\emph{lc centre}. The \emph{non-klt locus} $\Nklt(X,B)$ of a sub-pair $(X,B)$ is the union of 
the non-klt centres with reduced structure.

A sub-pair $(X,B)$ is \emph{log smooth} if $X$ is smooth and $\Supp B$ has simple 
normal crossing singularities.

\subsection{Minimal model program (MMP)}\label{ss-MMP} 
We will use standard results of the minimal model program (cf. [\ref{kollar-mori}][\ref{BCHM}]). 
Assume $(X,B)$ is a pair, $X\to Z$ is a projective morphism,  
$H$ is an ample$/Z$ $\R$-divisor, and that $K_X+B+H$ is nef$/Z$. 
First suppose $(X,B)$ is   
$\Q$-factorial dlt. Then we can run an MMP$/Z$ on $K_X+B$ with scaling of $H$ 
(cf. [\ref{B-mmodel}, \S 3]). 
In general we do not know whether the MMP terminates but 
we know that in some step of the MMP we reach a model $Y$ on which $K_Y+B_Y$, 
the pushdown of $K_X+B$, is a limit of movable$/Z$ $\R$-divisors: indeed, if the MMP terminates, then 
the claim is obvious; otherwise the MMP produces an infinite sequence $X_i\bir X_{i+1}$ 
of flips and a decreasing sequence $\lambda_i$ of scaling numbers in $(0,1]$ such that 
$K_{X_i}+B_i+\lambda_iH_i$ is nef$/Z$; by [\ref{BCHM}][\ref{B-lc-flips}, Theorem 1.9], $\lim\lambda_i=0$; 
in particular, if $Y:=X_1$, then $K_Y+B_Y$ is the limit of the movable$/Z$ $\R$-divisors 
$K_Y+B_Y+\lambda_i H_Y$. 

Now assume $(X,B)$ is klt with either $K_X+B$ or $B$ big$/Z$ (but $X$ not necessarily $\Q$-factorial). 
Then again we can run an MMP$/Z$ on $K_X+B$ with scaling of $H$ and this time the MMP terminates. 
We explain how this works. 

Let $X_1=X, B_1=B, H_1=H$ and let $\lambda_1$ be the smallest number such that 
$K_{X_1}+B_1+\lambda_1H_1$ is nef$/Z$. We can assume $\lambda_1>0$ otherwise the MMP 
ends on $X$. Moreover, we can choose $H_1\ge 0 $ in its $\R$-linear equivalence class so that 
$(X_1,B_1+\lambda_1H_1)$ is klt. Then $K_{X_1}+B_1+\lambda_1 H_1$ is semi-ample 
over $Z$ by the base point free theorem, 
hence it defines a contraction $X_1\to Z_1/Z$. 
Now $X_1$ has only finitely many extremal rays 
over $Z_1$ by the cone theorem 
[\ref{kollar-mori}, Theorem 3.25] as $H_1$ is big over $Z$ and $K_{X_1}+B_1+\lambda_1 H_1\equiv_{Z_1} 0$ 
(here $X_1$ is of Fano type over $Z_1$; see below for Fano type varieties). 
Since $\lambda_1>0$, there is an extremal ray $R_1$ over $Z_1$ such that 
$(K_{X_1}+B_1)\cdot R_1<0$. By the cone theorem, $R_1$ can be contracted via 
an extremal contraction $X_1\to T_1/Z_1$. If this contraction is not birational, 
then we get a Mori fibre space and the MMP ends here. 
Otherwise, $(X_1,B_1)$ has an lc model over $T_1$, say $(X_2,B_2)$ [\ref{BCHM}, Theorem 1.2]. 
Let $H_2$ be the pushdown of $H_1$ and let $\lambda_2$ be the smallest number such that 
$K_{X_2}+B_2+\lambda_2H_2$ is nef$/Z$. Now we repeat 
the same process on $X_2$ and so on to get the desired MMP. 
Note that $H_i$ may not be ample for $i>1$ but this is not a problem as 
we do not need ampleness beyond $X_1$.

We argue that the MMP terminates. We can assume that after finitely many steps the extremal contractions 
$X_i\to T_i$ are all flipping contractions. Pick $i\gg 0 $ and take a small $\Q$-factorialisation 
of $X_i$, say $X_i'\to X_i$. Let $B_i',H_i'$ be the birational transforms of $B_i,H_i$. 
Running an MMP on $K_{X_i'}+B_i'$ over $T_i$ with scaling of $\lambda_i H_i'$ lifts 
$X_i\bir X_{i+1}$ to a sequence of flips $X_i'\bir X_{i+1}'$ in the $\Q$-factorial case 
where $X_{i+1}'\to X_{i+1}$ is a small $\Q$-factorialisation. 
Putting all these together gives a sequence of flips for an MMP with scaling (of a big divisor) in the 
$\Q$-factorial case which terminates by [\ref{BCHM}, Corollary 1.4.2].  

Similar remarks apply to the MMP for generalised pairs defined below.

\subsection{Fano pairs}\label{ss-Fano}
Let $(X,B)$ be a pair and $X\to Z$ a contraction. We say $(X,B)$ is \emph{log Fano} over $Z$ 
if it is lc and $-(K_X+B)$ is ample over $Z$; if $B=0$ 
we just say $X$ is Fano over $Z$.
We say $X$ is \emph{of Fano type} over $Z$ if $(X,B)$ is klt log Fano over $Z$ for some choice of $B$;
it is easy to see this is equivalent to existence of a big$/Z$ $\Q$-boundary (resp. $\R$-boundary) 
$\Gamma$ so that $(X,\Gamma)$ is klt and $K_X+\Gamma \sim_{\Q,Z} 0$ (resp. $\sim_{\R,Z}$ instead of $\sim_{\Q,Z}$).

Assume $X$ is of Fano type over $Z$. Then we can run an MMP 
over $Z$ on any $\R$-Cartier $\R$-divisor $D$ on $X$ which ends with some model $Y$: 
the MMP is just an MMP on $K_X+\Gamma+tD$ with scaling of some ample divisor for some small $t>0$, as in \ref{ss-MMP}. 
If $D_Y$ is nef over $Z$,
we call $Y$ a \emph{minimal model} over $Z$ for $D$. If $D_Y$ is not nef$/Z$, then 
there is a $D_Y$-negative extremal contraction $Y\to T/Z$ with $\dim Y>\dim T$ and we call 
$Y$ a \emph{Mori fibre space} over $Z$ for $D$. 

If $X$ is of Fano type over $Z$ and $D$ is a nef$/Z$ $\R$-divisor on $X$, then $D$ is semi-ample over 
$Z$: there is a boundary $\Gamma$ so that $(X,\Gamma)$ is klt and $K_X+\Gamma \sim_{\R,Z} 0$; 
so choosing a small $t>0$, $(X,\Gamma+tD)$ is klt and $tD\sim_{\R,Z} K_X+\Gamma+tD$, hence $D$ is semi-ample over $Z$ 
because $K_X+\Gamma+tD$ is semi-ample over $Z$ by the base point free theorem.

\subsection{b-divisors}\label{ss-b-divisor}

We recall some definitions regarding b-divisors. 
Let $X$ be a normal variety. A \emph{b-divisor} $\mathbf{M}$ over $X$ is a collection of $\R$-divisors 
$M_Y$ on $Y$ for each birational contraction $Y\to X$ from a normal variety that are compatible 
with respect to pushdown, 
that is, if $Y'\to X$ is another birational contraction and $\psi\colon Y'\bir Y$ is a moprhism, 
then $\psi_*M_{Y'}=M_Y$. 
 
A b-divisor $\mathbf{M}$ is \emph{b-$\R$-Cartier} if there is a birational contraction $Y\to X$ 
such that $M_Y$ is $\R$-Cartier and such that for any birational contraction 
$\psi\colon Y'\to Y/X$ we have $M_{Y'}=\psi^*M_Y$. 
In other words, a b-$\R$-Cartier b-divisor over $X$ is determined by the choice of a birational contraction  
$Y\to X$ and an $\R$-Cartier $\R$-divisor $M$ on $Y$. But this choice is not unique,  
that is, another birational contraction $Y'\to X$ and an $\R$-Cartier $\R$-divisor
$M'$ on $Y'$ defines the same b-$\R$-Cartier b-divisor if there is a common resolution $W\to Y$ and $W\to Y'$ 
on which the pullbacks of $M$ and $M'$ coincide.  

A b-$\R$-Cartier  b-divisor  represented by some $Y\to X$ and $M$ is \emph{b-Cartier} if  $M$ is 
b-Cartier, i.e. its pullback to some resolution is Cartier.

\subsection{Generalised pairs}\label{ss-gpp}
For the basic theory of generalised polarised pairs see [\ref{BZh}, Section 4].
Below we recall some of the main notions and discuss some basic properties.\

(1)
A \emph{generalised pair} consists of 
\begin{itemize}
\item a normal variety $X$ equipped with a projective
morphism $X\to Z$, 

\item an $\R$-divisor $B\ge 0$ on $X$, and 

\item a b-$\R$-Cartier  b-divisor over $X$ represented 
by some projective birational morphism $X' \overset{\phi}\to X$ and $\R$-Cartier divisor
$M'$ on $X'$
\end{itemize}
such that $M'$ is nef$/Z$ and $K_{X}+B+M$ is $\R$-Cartier,
where $M := \phi_*M'$. 

We usually refer to the pair by saying $(X,B+M)$ is a  generalised pair with 
data $X'\overset{\phi}\to X\to Z$ and $M'$. Since a b-$\R$-Cartier b-divisor is defined birationally, 
in practice we will often replace $X'$ with a resolution and replace $M'$ with its pullback.
When $Z$ is not relevant we usually drop it
 and do not mention it: in this case one can just assume $X\to Z$ is the identity. 
When $Z$ is a point we also drop it but say the pair is projective. 

Now we define generalised singularities.
Replacing $X'$ we can assume $\phi$ is a log resolution of $(X,B)$. We can write 
$$
K_{X'}+B'+M'=\phi^*(K_{X}+B+M)
$$
for some uniquely determined $B'$. For a prime divisor $D$ on $X'$ the \emph{generalised log discrepancy} 
$a(D,X,B+M)$ is defined to be $1-\mu_DB'$. 

We say $(X,B+M)$ is 
\emph{generalised lc} (resp. \emph{generalised klt})(resp. \emph{generalised $\epsilon$-lc}) 
if for each $D$ the generalised log discrepancy $a(D,X,B+M)$ is $\ge 0$ (resp. $>0$)(resp. $\ge \epsilon$).
A \emph{generalised non-klt place} of $(X,B+M)$ is a prime divisor 
$D$ on birational models of $X$ with $a(D,X,B+M)\le 0$, and 
a \emph{generalised non-klt centre} of $(X,B+M)$ is the image of a generalised non-klt place.
The \emph{generalised non-klt locus} $\Nklt(X,B+M)$ of the generalised pair is the union of all 
the generalised non-klt centres.  

We will also use similar definitions when $B$ is not necessarily effective in which case we have a 
generalised sub-pair. 

(2)
Let $(X,B+M)$ be a generalised pair as in (1).
We say $(X,B+M)$ is \emph{generalised dlt} if it is generalised lc and if 
$\eta$ is the generic point of any generalised non-klt centre of 
$(X,B+M)$, then $(X,B)$ is log smooth near $\eta$ and $M'=\phi^*M$  holds over a neighbourhood of $\eta$.  
Note that when $M'=0$, then $(X,B)$ is {generalised dlt} iff it is dlt in the usual sense. 

The generalised dlt property is preserved under the MMP. Indeed, assume $(X,B+M)$ is {generalised dlt} and 
that $X\bir X''/Z$ is a divisorial contraction or a flip with respect to $K_{X}+B+M$. Replacing $\phi$ 
we can assume $X'\bir X''$ is a morphism. Let $B'',M''$ be the pushdowns of $B,M$ and consider 
$(X'',B''+M'')$ as a generalised pair with data $X'\to X''\to Z$ and $M'$.  Then $(X'',B''+M'')$
is also generalised dlt because it is generalised lc and because $X\bir X''$ is an isomorphism over the generic point of 
any generalised non-klt center of $(X'',B''+M'')$.

(3)
Let $(X,B+M)$ be a generalised pair as in (1) and let $\psi\colon X''\to X$ be a projective birational 
morphism from a normal variety. Replacing $\phi$ we can assume $\phi$ factors through 
$\psi$. We then let $B''$ and $M''$ be the pushdowns of 
$B'$ and $M'$ on $X''$ respectively. In particular,  
$$
K_{X''}+B''+M''=\psi^*(K_{X}+B+M).
$$
If $B''\ge 0$, then $(X'',B''+M'')$ is also a generalised pair 
with data $X'\to X''\to Z$ and $M'$. 

Assume that we can write $B''=\Delta''+G''$ where $(X'',\Delta''+M'')$ is $\Q$-factorial generalised dlt, 
$G''\ge 0$ is supported in $\rddown{\Delta''}$, and 
every exceptional prime divisor of $\psi$ is a component of $\rddown{\Delta''}$. Then we say 
$(X'',\Delta''+M'')$ is a \emph{$\Q$-factorial generalised dlt model} of $(X,B+M)$. Such models exist by 
the next lemma (also see [\ref{HMX2}, Proposition 3.3.1] and [\ref{BZh}, Lemma 4.5]). 
If $(X,B+M)$ is generalised lc, then $G''=0$.

\begin{lem}\label{l-Q-fact-dlt-model}
Let $(X,B+M)$ be a generalised pair with data $\phi\colon X'\to X$ and $M'$. 
Then the pair has a $\Q$-factorial generalised dlt model.
\end{lem}
\begin{proof}
Replacing $\phi$ we can assume it is a log resolution. Write 
$$
K_{X'}+B'+M'=\phi^*(K_{X}+B+M).
$$
Write $B=\Delta+G$ where $\Delta$ is obtained from $B$ by replacing each coefficient $>1$ with $1$. 
So $G\ge 0$ is supported in $\rddown{\Delta}$. 
Let $\Delta'$ be the sum of the birational transform of $\Delta$ and the reduced exceptional divisor 
of $\phi$. Let $G':=B'-\Delta'$. 

Run an MMP on $K_{X'}+\Delta'+M'$ over $X$ with scaling of an ample divisor. 
We reach a model $X''$ on which $K_{X''}+\Delta''+M''$ is a limit of movable$/X$ $\R$-divisors. 
By construction, 
$$
K_{X''}+\Delta''+M''+G''\equiv_X 0.
$$
So for any exceptional$/X$ prime divisor $S''$ on $X''$, $-G''|_{S''}$ is pseudo-effective over the image of $S''$ 
in $X$.
Therefore, by the general negativity lemma [\ref{B-lc-flips}, Lemma 3.3], $G''\ge 0$ (note that 
[\ref{B-lc-flips}, Lemma 3.3] implicitly assumes the base field is uncountable; 
if in our case the base field is countable, then we do a base change and then apply the lemma 
as in that case for the very general curves $C''$ of $S''$ 
contracted over $X$, we have $-G''\cdot C\ge 0$).

By definition of $\Delta'$, each exceptional prime divisor of $X''\to X$ is a component of 
$\rddown{\Delta''}$. Moreover, each component of $G''$ is either exceptional in which case it is a 
component of $\rddown{\Delta''}$, or non-exceptional in which case it is the birational transform 
of a component of $G$ hence again a component of $\rddown{\Delta''}$. Therefore, 
$(X'',\Delta''+M'')$ is a {$\Q$-factorial generalised dlt model} of $(X,B+M)$.

\end{proof}

(4) 
Next we prove two lemmas about generalised pairs that will be used later in Section 4.

\begin{lem}\label{l-application-of-local-global-ACC}
Let $d,p$ be natural numbers and $\Phi$ be a DCC set of non-negative real numbers. 
Then there is a positive real number $t>0$ depending only on $d,p,\Phi$ satisfying the following. 
Assume that 
\begin{itemize}
\item $(X,B+M)$ is a $\Q$-factorial generalised pair of dimension $d$ with data \\ $X'\to X\to Z$ and 
$M'$, 
\item $D\ge 0$ is an $\R$-divisor, 
\item the coefficients of $B,D$ are in $\Phi$ and $pM'$ is Cartier.
\end{itemize}
 Then we have:
\begin{enumerate}
\item if $(X,B+(1-t)D+M)$ is generalised lc, then $(X,B+D+M)$ is generalised lc, and 

\item if $(X,B+(1-t)D+M)$ is generalised lc, $X\to Z$ is a Mori fibre space for $(X,B+(1-t)D+M)$, 
and $K_X+B+D+M$ is nef over $Z$, then 
$$
K_X+B+D+M\equiv_Z 0.
$$
\end{enumerate}
\end{lem}
\begin{proof}
(1) If this is not true, then there exist a strictly decreasing sequence 
$t_i$ of real numbers approaching $0$ and a sequence $(X_i,B_i+M_i),D_i$ 
of pairs and divisors as in the lemma such that 
$$
(X_i,B_i+(1-t_i)D_i+M_i)
$$ 
is generalised lc but $(X_i,B_i+D_i+M_i)$ is not generalised lc. 
Let $u_i$ be the generalised lc threshold of $D_i$ with respect to $(X_i,B_i+M_i)$. 
Then $u_i$ belongs to an ACC set depending only on $d,p,\Phi$, by [\ref{BZh}, Theorem 1.5]. 
On the other hand, $1>u_i\ge 1-t_i$, so the $u_i$ form a sequence approaching $1$. 
This contradicts the ACC property.  

(2)
Now we prove the second claim. Again if it is not true, then there 
exist a strictly decreasing sequence 
$t_i$ of real numbers approaching $0$ and a sequence $(X_i,B_i+M_i),D_i$ 
of pairs and divisors as in the lemma such that 
$$
(X_i,B_i+(1-t_i)D_i+M_i)
$$ 
is generalised lc, $X_i\to Z_i$ a Mori fibre space structure, and 
with $K_{X_i}+B_i+D_i+M_i$ nef over $Z_i$ 
but such that 
$$
K_{X_i}+B_i+D_i+M_i\not\equiv_{Z_i} 0.
$$ 
Let $v_i\in (1-t_i,1)$ be the number such that 
$$
K_{X_i}+B_i+v_iD_i+M_i\equiv_{Z_i} 0.
$$ 

By (1), $(X_i,B_i+D_i+M_i)$ is generalised lc. 
Let $F_i$ be a general fibre of $X_i\to Z_i$. 
Restricting to $F_i$, we get 
$$
K_{F_i}+B_{F_i}+v_iD_{F_i}+M_{F_i}\equiv 0
$$ 
where $({F_i},B_{F_i}+v_iD_{F_i}+M_{F_i})$ naturally inherits the structure of a generalised 
pair, induced by  $({X_i},B_i+v_iD_i+M_i)$, with nef part $M_{F_i'}=M'_i|_{F_i'}$ where 
$F_i'$ is the fibre of $X_i'\to Z_i$ corresponding to $F_i$. Now the coefficients of $B_{F_i},D_{F_i}$ 
belong to $\Phi$ and $pM_{F_i'}$ is Cartier. Then we get a contradiction, by 
the global ACC [\ref{BZh},Theorem 1.6] 
as the coefficients of $B_{F_i}+v_iD_{F_i}$ are in a DCC but not finite set.

\end{proof}

\begin{lem}\label{l-tau-for-non-canonical-places-extraction}
Assume that 
\begin{itemize}
\item $(X,B+M)$ is a generalised lc pair with data $X'\to X$ and $M'$,
\item $(X,C+N)$ is generalised klt with data $X'\to X$ and $N'$, and  
\item $S$ is a prime divisor over $X$ with 
$$
a(S,X,B+M)<1.
$$
\end{itemize}
Then there is a birational contraction $Y\to X$ from a normal variety such that 
$S$ is a divisor on $Y$ and $-S$ is ample over $X$.
\end{lem}
\begin{proof}
Take a small rational number $u>0$ and consider the generalised pair 
$$
(X,uC+(1-u)B+uN+(1-u)M)
$$
with nef part $uN'+(1-u)M'$. The pair is generalised klt and 
$$
a(S,X,uC+(1-u)B+uN+(1-u)M)<1.
$$
Thus replacing $(X,B+M)$ with the pair above, we can assume that $(X,B+M)$ is generalised klt.

There exist an ample divisor $A'$ and an effective divisor $G'$ on $X'$ such that $A'+G'\sim_{\Q,X} 0$. 
Take a small  number $t>0$ and general element $0\le L'\sim_\R M'+tA'$. Letting 
$L$ be the pushdown of $L'$ we see that $(X,B+tG+L)$ is klt and 
$$
a(S,X,B+tG+L)<1.
$$ 
Thus replacing $(X,B+M)$ with $(X,B+tG+L)$ we can assume 
$M'=0$ and that $(X,B)$ is klt. 

Now by [\ref{BCHM}, Corollary 1.4.3], there is a crepant $\Q$-factorial 
terminal model $(V,B_V)$ of $(X,B)$. 
Then $S$ is a divisor on $V$ and the coefficient of $S$ in 
$B_V$ is positive. Since 
$$
K_V+B_V-cS\sim_{\R,X} -cS,
$$ 
where $c>0$ is sufficiently small, 
we see that $-S$ has an ample model $Y$ over $X$.
By the negativity lemma, $S$ is not contracted over $Y$. Abusing notation we denote 
the birational transform of $S$ on $Y$ again by $S$. 
Then on $Y$, $-S$ is ample over $X$,  so we get the desired model.
 
\end{proof}

\subsection{Generalised adjunction for fibrations} \label{ss-Adjunction-gen-pair}
 Consider the following set-up. Assume that
\begin{itemize}
  \item $(X,B+M)$ is a generalised sub-pair with data $X'\rightarrow X\to Z$ and $M'$,
  \item $f:X\rightarrow Z$ is a contraction with $\dim Z>0$,
  \item $(X,B+M)$ is generalised sub-lc over the generic point of $Z$, and
  \item $K_X+B+M\sim_{\mathbb{R},Z}0$.
\end{itemize}
We define the discriminant divisor $B_Z$ for the above setting. Let $D$ be a prime divisor on $Z$.
Let $t$ be the largest real number such that $(X,B+tf^*D+M)$ is generalised sub-lc
over the generic point of $D$. This makes sense even if $D$ is not $\mathbb{Q}$-Cartier because
we only need the pullback $f^*D$ over the generic point of $D$ where $Z$ is smooth.
 We then put the coefficient of $D$ in $B_Z$ to be $1-t$. Note that since $(X,B+M)$ is generalised
 sub-lc over the generic point of $Z$, $t$ is a real number, that is, it is not $-\infty$ or $+\infty$.
 Having defined $B_Z$, we can find $M_Z$ giving
$$K_X+B+M\sim_{\mathbb{R}}f^*(K_Z+B_Z+M_Z)$$
where $M_Z$ is determined up to $\mathbb{R}$-linear equivalence. We call $B_Z$ the 
\emph{discriminant divisor of adjunction} for $(X,B+M)$ over $Z$. 
If $B,M'$ are $\Q$-divisors and $K_X+B+M\sim_{\mathbb{Q},Z}0$, then
$B_Z$ is a $\Q$-divisor and we can choose $M_Z$ also to be a $\Q$-divisor.
For any birational morphism $Z'\to Z$ from a normal variety, we can similarly define $B_{Z'}$ and $M_{Z'}$. This gives the discriminant and moduli b-divisors.

For more details about generalised adjunction for  fibrations, we refer to [\ref{Filipazzi-18}] and
\S 6.1 of [\ref{B-lcyf}]. 

\begin{thm}
Under the above notation, assume that $X$ is projective,
$(X,B+M)$ is generalised lc over the generic point of $Z$, and $B,M'$ are
$\Q$-divisors, and $M'$ is globally nef. 
Then the moduli divisor of adjunction is nef b-$\Q$-Cartier, that is, 
there is a resolution $Z'\to Z$ such that  
$M_{Z'}$ is nef $\Q$-Cartier and for any birational contraction 
$Z''\to Z'$, $M_{Z''}$ is the pullback of $M_{Z'}$.  
\end{thm}

In particular, we can regard $(Z,B_Z+M_Z)$ as a generalised sub-pair with nef part $M_{Z'}$. 
The theorem is proved in [\ref{ambro-adj}, \S 3] (based on [\ref{kaw-subadjuntion}]) when $M'=0$, and in 
[\ref{Filipazzi-18}, Theorem 1.4] in general. We will use the theorem in the proof of Theorem \ref{t-mirror-symmetry}.

\subsection{Connected components and \'etale neighbourhoods}\label{ss-etale}
Let $Z$ be a variety and $g\colon N\to Z$ be a projective morphism where $N$ is a scheme. 
In the discussion below keep in mind that the topology of the fibre of $g$ over a 
point $z\in Z$ is the same as the subset topology on $g^{-1}\{z\}$ induced by the 
topology on $N$ where $g^{-1}\{z\}$ means the set-theoretic inverse image 
(a similar fact holds in general for any morphism of schemes). 

(1)
We show that the fibres of $g$ are connected iff its fibres 
over closed points are connected. We actually prove a stronger statement. 
Let $z\in Z$ and let $R$ be the closure of $z$ in $Z$. We claim that   
for any closed point $y$ in some non-empty open subset of $R$, 
the number of connected components of $g^{-1}\{y\}$ is at least the 
number of connected components of $g^{-1}\{z\}$. 

To prove the claim, we can replace $Z$ with $R$ and replace $N$ with the scheme-theoretic 
inverse image of $R$, hence assume $z$ is the generic point of $Z$.  
Let $C_1, \dots, C_r$ be the connected components of $g^{-1}\{z\}$, and let $\overline{C}_i$ 
be the closure of $C_i$ in $N$. Then the fibre of $\overline{C}_i\to Z$ over $z$ is just $C_i$. 
Therefore, $\overline{C}_i\cap \overline{C}_j$ does not surject onto 
$Z$ for any $i\neq j$ otherwise $C_i$ and $C_j$ would intersect, 
so shrinking $Z$ we can assume that $\overline{C}_i\cap \overline{C}_j=\emptyset$ 
for $i\neq j$. But then since $\overline{C}_i$ surjects onto $Z$ for each $i$, the number of 
connected components of $g^{-1}\{y\}$ is at least $r$, for any closed point $y\in Z$. 
This proves the claim. 

(2) 
For each closed point $z\in Z$, there is an \'etale neighbourhood $\tilde{Z}\to Z$ with a closed point $\tilde{z}$ mapping 
to $z$ such that there is a 1-1 correspondence between the connected components of 
$\tilde{N}:=N\times_Z\tilde{Z}$, and the connected components of the fibre of $g$ over $z$ [\ref{Kollar-singMMP}, claim 4.38.1].    
If the ground field is $\C$, then this essentially says that the connected components of $N$ 
over a small analytic neighbourhood of $z$ correspond to the connected components of the fibre 
$N\to Z$ over $z$.


\section{\bf Non-klt loci of anti-nef pairs}

In this section we prove Theorems \ref{t-anti-nef-pairs-1} and \ref{t-anti-nef-pairs-2} 
in the more general framework of generalised pairs. Using generalised pairs is important for the proofs even 
if one is only interested in usual pairs.

\begin{thm}\label{t-anti-nef-pairs-2-generalised}
Let $(X,B+M)$ be a generalised pair with data $X'\to X\to Z$ and $M'$ where $f\colon X\to Z$ is a contraction. 
Assume $-(K_X+B+M)$ is nef over $Z$. 
Then the fibres of 
$$
\Nklt(X,B+M)\to Z
$$ 
are connected if any of the following conditions holds:  
\begin{enumerate}
\item $-(K_X+B+M)$ is big over $Z$, or 

\item $\Nklt(X,B+M)\to Z$ is not surjective, or 

\item $\tau_S(-(K_X+B+M)/Z)>0$ for every generalised non-klt place $S$ of $(X,B+M)$.
\end{enumerate}
\end{thm}

We prove cases (1) and (2) first and then prove case (3) towards the end of this section.
Note that in view of Theorem \ref{t-anti-nef-pairs-1-generalised} below, we can add another case 
in which the theorem holds, that is, the fibre of 
$\Nklt(X,B+M)\to Z$ over a point $z\in Z$ is connected if $(X,B+M)$ is not generalised lc over $z$.  
 
\begin{rem}\label{rem-connected-near-fib}
\emph{ 
In view of \ref{ss-etale}, to prove the theorem, 
it is enough to prove the weaker statement that $\Nklt(X,B+M)$ is connected near each fibre of $X\to Z$  
as all the conditions (1)-(3) are preserved after \'etale base change.
We explain this point in detail. 
By \ref{ss-etale}(1), it is enough to show that the fibres of 
$$
N:=\Nklt(X,B+M)\to Z
$$ 
over closed points 
are connected. And by \ref{ss-etale}(2), for each 
closed point $z\in Z$, there is an \'etale neighbourhood $\tilde{Z}\to Z$ with a closed point $\tilde{z}$ mapping 
to $z$ such that there is a 1-1 correspondence between the connected components of 
$\tilde{N}:=N\times_Z\tilde{Z}$, and the connected components of the fibre of $N\to Z$ over $z$. 
Fibre product with $\tilde{Z}$ induces a generalised pair $(\tilde{X},\tilde{B}+\tilde{M})$ 
with data $\tilde{X}'\to \tilde{X}\to \tilde{Z}$ and $\tilde{M'}$ where 
$\tilde{N}=\Nklt(\tilde{X},\tilde{B}+\tilde{M})$. It is enough to show that 
$\Nklt(\tilde{X},\tilde{B}+\tilde{M})$ is connected near the fibre of $\tilde{X}\to \tilde{Z}$ over $\tilde{z}$, 
hence replacing $(X,B+M)$, $X'\to X\to Z$ and $M'$ with $(\tilde{X},\tilde{B}+\tilde{M})$,  
$\tilde{X}'\to \tilde{X}\to \tilde{Z}$ and $\tilde{M'}$, respectively, it is enough to show that 
$\Nklt(X,B+M)$ is connected near each fibre of $X\to Z$.}
\end{rem}

\begin{lem}\label{l-anti-nef-pairs-2-generalised-(1)}
Theorem \ref{t-anti-nef-pairs-2-generalised} (1) holds. 
\end{lem}
\begin{proof}
As pointed out above it is enough to show that 
$\Nklt(X,B+M)$ is connected near each fibre of $X\to Z$. 
We can then use [\ref{B-compl}, Lemma 2.14] which reduces the statement to the connectedness principle 
for usual pairs.

\end{proof}

\begin{lem}\label{l-anti-nef-pairs-2-generalised-(2)}
Theorem \ref{t-anti-nef-pairs-2-generalised} (2) holds. 
\end{lem}
\begin{proof}
\emph{Step 1.}
In this step we reduce the statement to the case $K_X+B+M\equiv_Z 0$.
By Remark \ref{rem-connected-near-fib}, it is enough to show that 
$\Nklt(X,B+M)$ is connected near each fibre of $X\to Z$ over closed points. 
Assume that $\Nklt(X,B+M)$ is not connected near the fibre of $X\to Z$ over some closed point $z$. 
Shrinking $Z$ around $z$, we can assume that $\Nklt(X,B+M)$ is not connected globally. 
Extending the ground field we can assume it is not countable.

Let 
$$
L:=-(K_X+B+M)
$$ 
and let $L'$ be the pullback of $L$ on $X'$. Since $L$ is nef over $Z$, 
$L'$ is nef over $Z$, hence $M'+L'$ is nef over $Z$.
Consider the generalised pair $(X,B+M+L)$ with data $X'\to X\to Z$ and $M'+L'$. The generalised non-klt locus of 
$(X,B+M+L)$ coincides with that of $(X,B+M)$ because $L$ being nef over $Z$ means that 
for any prime divisor $S$ over $X$ we have 
$$
a(S,X,B+M+L)=a(S,X,B+M).
$$
Thus replacing $M'$ with $M'+L'$ we can assume 
that $K_X+B+M\equiv_Z 0$.\\ 

\emph{Step 2.}
In this step we modify $(X,B+M)$ so that $B=\Delta+G$ where $B,G\ge 0$, $G$ is supported in 
$\rddown{\Delta}$, and 
$(X,\Delta-t\rddown{\Delta}+M)$ is $\Q$-factorial generalised klt for some $t\in (0,1)$.
Let $(X'',\Delta''+M'')$ be a $\Q$-factorial generalised dlt model of $(X,B+M)$ which exists by Lemma \ref{l-Q-fact-dlt-model}. By definition, $(X'',\Delta''+M'')$ is 
$\Q$-factorial generalised dlt.
Denoting $X''\to X$ by $\psi$, we have 
$$
K_{X''}+\Delta''+G''+M''=\psi^*(K_X+B+M)
$$
where $G''\ge 0$ is supported in $\rddown{\Delta''}$. 
We can assume $X'\bir X''$ is a morphism, so we can consider $(X'',\Delta''+G''+M'')$ 
as a generalised pair with nef part $M'$. Since 
$$
\Nklt(X,B+M)=\psi(\Nklt(X'',\Delta''+G''+M'')),
$$
we deduce that 
$$
\Nklt(X'',\Delta''+G''+M'')
$$ 
is not connected over $z$. Thus 
we can replace $(X,B+M)$ with $(X'',\Delta''+G''+M'')$, hence 
we can assume that the following condition holds: 

\begin{itemize}
\item[$(*)$] $B=\Delta+G$ where $G\ge 0$ is supported in $\rddown{\Delta}$ and 
$(X,\Delta-t\rddown{\Delta}+M)$ is $\Q$-factorial generalised klt for some $t\in (0,1)$.
\end{itemize}
The condition implies that any generalised non-klt centre of $(X,B+M)$ is contained in the support of 
$ t\rddown{\Delta}+G$, so we have 
$$
\rddown{\Delta}\subseteq \Nklt(X,B+M)\subseteq \Supp (t\rddown{\Delta}+G)=\rddown{\Delta},
$$
giving
$$
\Nklt(X,B+M)=\Supp (t\rddown{\Delta}+G)=\rddown{\Delta}.
$$
By assumption, $\Nklt(X,B+M)$ is vertical over $Z$, so $\rddown{\Delta}$ is vertical over $Z$. 
Thus $t\rddown{\Delta}+G$ is also vertical over $Z$, so 
$$
K_X+\Delta-t\rddown{\Delta}+M\equiv 0
$$
over some non-empty open subset of $Z$, hence $K_X+\Delta-t\rddown{\Delta}+M$ is pseudo-effective 
over $Z$.\\ 

\emph{Step 3.}
In this step we run an MMP on $K_X+\Delta-t\rddown{\Delta}+M$ and see that 
the condition $\Nklt(X,B+M)=\rddown{\Delta}$ is preserved by the MMP.  
Ineed, we can run an MMP on $K_X+\Delta-t\rddown{\Delta}+M$ over $Z$ with scaling of some 
ample divisor $H$ [\ref{BZh}, Lemma 4.4] but we do not claim that the MMP terminates. However, if $\lambda_i$ are 
the scaling numbers that appear in the MMP, then $\lim \lambda_i=0$, by [\ref{BZh}, Lemma 4.4].
Since 
$$
K_X+\Delta+G+M\equiv_Z 0,
$$ 
the divisor $t\rddown{\Delta}+G$ is numerically positive on the extremal ray of each step of 
the MMP. Clearly the condition $(*)$ is preserved by the MMP, so the property 
$$
\Nklt(X,B+M)=\Supp (t\rddown{\Delta}+G)=\rddown{\Delta}
$$
is also preserved. Also note that since $K_X+\Delta-t\rddown{\Delta}+M$ is pseudo-effective over $Z$, 
every step of the MMP is a divisorial contraction or a flip.\\

\emph{Step 4.}
In this step we show that $\Nklt(X,B+M)$ remains disconnected near the fibre over $z$ during the MMP.
More precisely, we show that there is a 1-1 correspondence (given by divisorial pushdown) 
between the connected components of 
$\Nklt(X,B+M)$ and of $\Nklt(X'',B''+M'')$ for any model $X''$ appearing in the MMP. 
By Step 3, $\Nklt(X,B+M)$ coincides with $\Supp (t\rddown{\Delta}+G)$, and 
$\Nklt(X'',B''+M'')$ coincides with $\Supp(t\rddown{\Delta}''+G'')$, so it is enough to 
prove the assertion for connected components of $\Supp (t\rddown{\Delta}+G)$ 
and of $\Supp(t\rddown{\Delta}''+G'')$.

Say $X\bir X''$ is the first step of the 
MMP which is either a divisorial contraction or a flip. First assume $X\bir X''$ is a flip and let 
$X\to V$ be the corresponding flipping contraction. 
Each connected component of $\Supp (t\rddown{\Delta''}+G'')$ away from the flipped locus is just the birational transform of a connected component of $\Supp (t\rddown{\Delta}+G)$ 
away from the flipping locus. 
On the other hand, since $t\rddown{\Delta}+G$ is ample over $V$, there is a connected component 
$\mathcal{C}$ of $\Supp (t\rddown{\Delta}+G)$ intersecting every positive-dimensional fibre of $X\to V$. 
By Step 2 and Lemma \ref{l-anti-nef-pairs-2-generalised-(1)}, $\Supp (t\rddown{\Delta}+G)$ is connected near any fibre of $X\to V$ as $-(K_X+B+M)$ is nef and big over $V$. Therefore, no other connected component of $\Supp (t\rddown{\Delta}+G)$  intersects the exceptional locus of $X\to V$. 
It is enough to show that the birational transform $\mathcal{C}''$ of $\mathcal{C}$ is connected. 
 
If a connected component of $\mathcal{C}''$ does not intersect the flipped locus, then 
it is the birational transform of a connected component of $\mathcal{C}$, hence 
it coincides with the whole $\mathcal{C}''$ (as $\mathcal{C}$ is connected), 
which is not possible because $\mathcal{C}''$ intersects the flipped locus. Thus every 
connected component of $\mathcal{C}''$ intersects the flipped locus. On the other hand,    
the birational transform $t\rddown{\Delta}''+G''$ of $t\rddown{\Delta}+G$ 
is anti-ample over $V$, so at least one of its irreducible 
components, say $D''$, contains the flipped locus. 
Thus $D''$ intersects every connected component of $\Supp(t\rddown{\Delta}''+G'')$ near the flipped locus. 
Therefore, $\Supp(t\rddown{\Delta}''+G'')$ 
has only one connected component near the flipped locus, and this connected component is exactly $\mathcal{C}''$. 
The claim is then proved in the flip case.

A similar argument shows that if $X\bir X''$ is a divisorial contraction, then exactly one 
connected component $\mathcal{C}$ of $\Supp (t\rddown{\Delta}+G)$  
intersects the exceptional divisor. Moreover, in this 
case $\mathcal{C}$ is not contracted by $X\bir X''$, that is, $\mathcal{C}$ is not 
equal to the exceptional divisor, by the negativity lemma, because 
$t\rddown{\Delta}+G$ is ample over $X''$. Thus each connected component of $\Supp(t\rddown{\Delta}''+G'')$ 
is just the divisorial pushdown of exactly one connected component of $\Supp (t\rddown{\Delta}+G)$. 
This proves the claim in the divisorial contraction case. 

Since the claim holds in each step of the MMP, it holds on any model appearing in the MMP.\\

\emph{Step 5.}
In this step we show that after finitely many steps of the MMP, 
the irreducible components of the fibre of $X\to Z$ over $z$ not contained in the non-klt locus 
$\Nklt(X,B+M)$ stabilise.
Let $F_1,\dots,F_r$ be the irreducible components of the fibre of $X\to Z$ over $z$. 
We claim that there is $i$ such that $F_i$ is not contained in $\Nklt(X,B+M)$ and that 
this is preserved by the MMP, that is, the birational transform of $F_i$ is not contained in 
the non-klt locus during the MMP. 

By assumption, $f\colon X\to Z$ is a contraction, so its fibre over $z$ is connected which 
means the set-theoretic inverse image $f^{-1}\{z\}$ is connected. On the other hand, 
we assumed that $\Nklt(X,B+M)$ is not connected near the fibre of $X\to Z$ over $z$, 
so 
$$
f^{-1}\{z\}\not\subseteq \Nklt(X,B+M)=\Supp (t\rddown{\Delta}+G)
$$
as $f^{-1}\{z\}$ intersects every connected component of $\Nklt(X,B+M)$ over $z$.
Thus some of the $F_i$ are not contained in $\Nklt(X,B+M)$; rearranging the indices we can assume that $F_1,\dots,F_s$ are not 
contained in $\Nklt(X,B+M)$ but $F_{s+1},\dots,F_r$ are contained. 

Let $X\bir X''$ be the first step of the MMP. Recall that $t\rddown{\Delta}+G$ is positive on the extremal 
ray of each step of the MMP. Thus 
if $X\bir X''$ is a divisorial contraction, then $\Nklt(X'',B''+M'')$ contains the image of 
the exceptional divisor, and if $X\bir X''$ is a flip, then similarly $\Nklt(X'',B''+M'')$ contains the  
flipped locus, that is, the exceptional locus of $X''\to V$ where $X\to V$ is the corresponding 
flipping contraction. This implies that any irreducible component of the fibre of $X''\to Z$ over $z$ 
which is not contained in $\Nklt(X'',B''+M'')$ is the birational transform of one of the $F_1,\dots,F_s$. 
Therefore, after finitely many steps, the irreducible components of the fibre over $z$ not contained in the non-klt locus 
stabilise, that is, replacing $X$ and possibly decreasing $s$ we can assume that on each model appearing in the MMP, the 
birational transforms of the $F_1,\dots,F_s$  are exactly the irreducible 
components of the fibre over $z$  not contained in the non-klt locus.\\

\emph{Step 6.}
In this step we show that 
$$
(K_X+\Delta-t\rddown{\Delta}+M)|_{F_j}
$$
is pseudo-effective for each $F_j$, $1\le j\le s$, where $F_1,\dots, F_s$ are 
as in the previous step. 
Since $F_j$ is not contained in $\Supp (t\rddown{\Delta}+G)$ in the course 
of the MMP, the map $X\bir X''$ is an isomorphism near the generic point of $F_j$ for any model $X''$ 
that appears in a step of the MMP. Since the MMP is an MMP on $K_X+\Delta-t\rddown{\Delta}+M$ 
over $Z$ with scaling of a big divisor, we deduce that $F_j$ is not contained in 
the relative stable base locus 
$$
{\bf B}(K_X+\Delta-t\rddown{\Delta}+M+\lambda_iH/Z)
$$ 
for any $i$ where $\lambda_i$ are the scaling numbers in the MMP. 
In particular,  this implies that 
$$
(K_X+\Delta-t\rddown{\Delta}+M+\lambda_iH)|_{F_j}
$$
is pseudo-effective for every $i$. But then since $\lim \lambda_i=0$,
$$
(K_X+\Delta-t\rddown{\Delta}+M)|_{F_j}
$$
is pseudo-effective. Therefore, 
$$
(K_X+\Delta-t\rddown{\Delta}+M)\cdot C\ge 0
$$
for any curve $C\subset F_j$ outside some countable union of subvarieties of $F_j$. 
Since the ground field is not countable, this countable union is not equal to $F_j$.\\

\emph{Step 7.}
In this step we finish the proof.
Since the fibre of $X\to Z$ over $z$ is connected, there is $1\le j\le s$
such that $F_j$ intersects $\Nklt(X,B+M)$. By the previous step, we can find a curve $C\subset F_j$ intersecting $\Nklt(X,B+M)$ but not contained in it such that  
$$
(K_X+\Delta-t\rddown{\Delta}+M)\cdot C\ge 0.
$$
Then $C$ intersects $\Supp(t\rddown{\Delta}+G)$ but is not contained in it. Therefore, 
$$
(t\rddown{\Delta}+G)\cdot C>0
$$
which contradicts  
$$
(K_X+\Delta-t\rddown{\Delta}+M+t\rddown{\Delta}+G)\cdot C=0.
$$

\end{proof}

Next we treat a generalised version of Theorem \ref{t-anti-nef-pairs-1}.

\begin{thm}\label{t-anti-nef-pairs-1-generalised}
Let  $(X,B+M)$ be a generalised pair with data $X'\to X\to Z$ where $f\colon X\to Z$ is a contraction. 
Assume $-(K_X+B+M)$ is nef over $Z$ and that the fibre of 
$$
g:\Nklt(X,B+M)\to Z
$$ 
over some point $z\in Z$ is not connected. Then we have:
\begin{enumerate}
\item $g$ is surjective and its fibre over $z$ has exactly two connected components,

\item the pair $(X,B+M)$ is generalised lc over $z$ and after base change to an \'etale neighbourhood of $z$ and 
replacing $X'$ with a high resolution, there exist a contraction 
$X'\to Y'/Z$ such that if 
$$
K_{X'}+B'+M':=\phi^*(K_X+B+M)
$$ 
and if $F'$ is a general fibre of $X'\to Y'$, 
then $M'|_{F'}\equiv 0$ and $(F',B'|_{F'})$ is isomorphic to $(\PP^1,p_1+p_2)$ for distinct points $p_1,p_2$.

Moreover, $\rddown{B'}$ has two disjoint irreducible components $S',T'$, both horizontal over $Y'$, and 
the images of $S',T'$ on $X$ are exactly the two connected components of $\Nklt(X,B+M)$. 
\end{enumerate}
\end{thm}
\begin{proof}
The proof is similar to the proof of [\ref{Kollar-singMMP}, Proposition 4.37] but with some 
crucial differences. 

\emph{Step 1.}
In this step we reduce the theorem to the case $K_X+B+M\equiv_Z 0$.
By Lemma \ref{l-anti-nef-pairs-2-generalised-(2)}, Theorem \ref{t-anti-nef-pairs-2-generalised}(2) holds, so $g$
is surjective as we are assuming that the fibre of $g$ over $z$ is not connected  
(the corresponding argument in [\ref{Kollar-singMMP}, Proposition 4.37] instead uses torsion freeness of 
certain higher direct image sheaves when $M'=0$ and $K_X+B\sim_{\Q,Z} 0$).

Let 
$$
L:=-(K_X+B+M)
$$ 
and let $L'$ be the pullback of $L$ on $X'$. 
Consider the generalised pair $(X,B+M+L)$ with nef part $M'+L'$. 
Writing 
$$
K_{X'}+B'+M':=\phi^*(K_X+B+M)
$$ 
we get
$$
K_{X'}+B'+M'+L':=\phi^*(K_X+B+M+L)
$$ 
meaning $B'$ is unchanged after adding $L$. The generalised log discrepancies 
of the two pairs $(X,B+M+L)$ and $(X,B+M)$ are equal, in particular, 
the non-klt locus of 
$(X,B+M+L)$ coincides with that of $(X,B+M)$. 
Thus replacing $M'$ with $M'+L'$ we can assume that $K_X+B+M\equiv_Z 0$.\\

\emph{Step 2.}
In this step we argue that to prove the theorem we can replace $Z$ with an \'etale neighbourhood of $z$.
Let $\tilde{Z}$ be an \'etale neighbourhood of $z$ with a point $\tilde{z}$ mapping to $z$. 
Fibre product with $\tilde{Z}$ gives a generalised pair $(\tilde{X},\tilde{B}+\tilde{M})$ 
with data $\tilde{X}'\to \tilde{X}\to \tilde{Z}$ and $\tilde{M'}$, and $\Nklt(\tilde{X},\tilde{B}+\tilde{M})$ 
is the inverse image of $\Nklt(X,B+M)$ under the morphism $\tilde{X}\to X$. If the fibre of 
$$
\Nklt(\tilde{X},\tilde{B}+\tilde{M})\to \tilde{Z}
$$
over $\tilde{z}$ has exactly two connected components, then the fibre of $g$ over $z$ also 
has exactly two connected components because the former fibre maps surjectively onto 
the latter fibre (note we already know the latter fibre is not connected).
Therefore, to prove (1) and (2) we are free to replace $Z$ with an \'etale neighbourhood of $z$.\\

\emph{Step 3.}
In this step we reduce the theorem to the case when $z$ is a closed point, and replace $Z$ with an  
\'etale neighbourhood of $z$.
Let $R$ be the closure of $z$ in $Z$. By \ref{ss-etale} (1),  
 the number of connected components of $g^{-1}\{v\}$ 
is at least the number of connected components of $g^{-1}\{z\}$ for the closed points $v$ in some 
non-empty open subset of $R$. 
Then to show that $g^{-1}\{z\}$ has exactly two connected components, it 
is enough to show that $g^{-1}\{v\}$ has exactly two connected components for a general 
closed point $v\in R$.

On the other hand, by \ref{ss-etale} (2), after base change to an \'etale 
neighbourhood of $v$ (which is automatically an \'etale 
neighbourhood of $z$), we can assume that $\Nklt(X,B+M)$ is not connected 
and that distinct connected components of $g^{-1}\{v\}$ are contained in 
distinct connected components of $\Nklt(X,B+M)$ (note that the number of connected 
components of $g^{-1}\{v\}$ is unchanged by the base change). We can then replace 
$z$ with $v$ and assume it is a closed point, and it is enough to prove (2) of the theorem 
without taking further \'etale base change.\\

\emph{Step 4.}
In this step we modify $(X,B+M)$ by taking a dlt model. Indeed,  
after taking a $\Q$-factorial generalised dlt model, as in Step 2 of the proof of 
Lemma \ref{l-anti-nef-pairs-2-generalised-(2)}, we can replace $X$ so that the following holds: 
\begin{itemize}
\item[$(*)$] $B=\Delta+G$ where $G\ge 0$ is supported in $\rddown{\Delta}$ and 
$(X,\Delta-t\rddown{\Delta}+M)$ is $\Q$-factorial generalised klt for some real number $t>0$. 
\end{itemize}
In particular, we have 
$$
\Nklt(X,B+M)=\Supp (t\rddown{\Delta}+G)=\rddown{\Delta}.
$$
Moreover, since some component of $\Nklt(X,B+M)$ is horizontal over $Z$, $t\rddown{\Delta}+G$ also  dominates $Z$.
Therefore, we see that 
$$
K_{X}+\Delta-t\rddown{\Delta}+M\equiv_Z -(t\rddown{\Delta}+G)
$$ 
is not pseudo-effective over $Z$.\\

\emph{Step 5.}
In this step we get a Mori fibre space for $({X},\Delta-t\rddown{\Delta}+M)$ and 
study its non-klt locus.
We can run an MMP on $K_{X}+\Delta-t\rddown{\Delta}+M$ over $Z$ ending with a Mori fibre space 
$X''\to Y'/Z$ [\ref{BZh}, Lemma 4.4]. Since $t\rddown{\Delta}+G$ is positive on the extremal ray in each step of the MMP, 
arguing as in Step 4 of the proof of Lemma \ref{l-anti-nef-pairs-2-generalised-(2)}, 
we see that the number of the connected components of 
$\Nklt(X,B+M)$ remains the same throughout the MMP. Moreover, 
 $t\rddown{\Delta''}+G''$ is ample over $Y'$. 
So at least one irreducible component $S''$ of $\rddown{\Delta''}$ is ample over $Y'$ which implies that $S''$ intersects 
every irreducible component of each fibre of $X''\to Y'$. In particular, if $T''$ is any vertical$/Y'$ component of $\rddown{\Delta''}$, then 
$S''$ intersects $T''$. 

Since 
$$
(X'',\Delta''-t\rddown{\Delta''}+M'')
$$ 
is generalised klt, 
$$
\Nklt(X'',B''+M'')=\Supp (t\rddown{\Delta''}+G'')=\rddown{\Delta''}.
$$
Thus $\rddown{\Delta''}$ has at least two connected components, hence every irreducible component 
of any connected component of $\rddown{\Delta''}$
is horizontal over $Y'$.\\

\emph{Step 6.}
In this step we show that the general fibres of $X''\to Y'$ are isomorphic to $\PP^1$ and 
further study $\Nklt(X'',B''+M'')$.
Let $\mathcal{C}_1'', \mathcal{C}_2''$ be two connected components of $\Nklt(X'',B''+M'')$ where  
$S''$ of the previous step is an irreducible component of $\mathcal{C}_1''$. 
Pick an irreducible component $T''$ of 
$\mathcal{C}_2''$. Let $F''$ be a general fibre of $X''\to Y'$. If $\dim T''\cap F''>0$, then $S''$ intersects 
$T''\cap F''$ as $S''$ is ample over $Y'$, which is not possible as $S''\cap T''=\emptyset$. Therefore, 
$\dim F''=1$, hence $F''\simeq \PP^1$. Since 
$$
(K_{X''}+\Delta''+G''+M'')|_{F''}\equiv 0
$$ 
and since at least two irreducible components $S'',T''$ of $\rddown{\Delta''}$ intersect $F''$ 
we see that near $F''$ we have 
$$
B''=\Delta''+G''=S''+T'',
$$ 
and  that $M''|_{F''}\equiv 0$. Therefore,  
$$
(F'',B''|_{F''})=(F'',(S''+T'')|_{F''})\simeq (\PP^1,p_1+p_2)
$$
where $p_1,p_2$ are two distinct points.\\
 
\emph{Step 7.}
In this step we show that $(X'',B''+M'')$ is generalised plt whose only generalised non-klt places are $S'',T''$. 
By the previous two steps,  $\mathcal{C}_1''=S'', \mathcal{C}_2''=T''$ are the only 
connected components of $\Nklt(X'',B''+M'')$ otherwise we would find a horizontal over $Y'$ component 
of $\rddown{\Delta''}$ other than $S'',T''$, which is not possible. Assume that $(X'',B''+M'')$ has another 
generalised non-klt place, say $R'''$. Consider $({X''},sB''+sM'')$ where $s$ is the biggest number so that the pair is 
generalised lc. Replacing $R'''$ we can assume $R'''$ is a generalised non-klt place of $({X''},sB''+sM'')$. 
Then we can extract $R'''$ say via an extremal contraction $X'''\to X''$. Note that $R'''$ is vertical over $Y'$.
Also $R'''$ maps into one of $S'',T''$, say $S''$, so $R'''$ does not intersect the birational transform $T'''$ of $T''$. 
 
Since $X''$ is of Fano type over $Y'$, so is $X'''$. Then we can run an MMP on $-R'''$ over $Y'$ ending with a good 
minimal model $X''''$. Since $R'''$ does not dominate $Y'$, $-R''''$ defines a contraction 
$X''''\to V'/Y'$ where $V'\to Y'$ is birational. Moreover, $R''''$ and $T''''$ are disjoint: 
indeed, if $K_{X'''}+B'''+M'''$ is the pullback of $K_{X''}+B''+M''$, then applying 
Theorem \ref{t-anti-nef-pairs-1-generalised} (1) (through Lemma \ref{l-anti-nef-pairs-2-generalised-(1)}), 
we see that at most one connected component of $\Nklt(X''',B'''+M''')$ intersects the exceptional 
locus of each step of the MMP and that component is the one containing $R'''$. 

However, $T''''$ is horizontal over $V'$ as 
$X''''\to V'$ and $X''\to Y'$ are the same over the generic point of $Y'$. But then $T''''$ intersects every fibre of 
$X''''\to V'$, hence it also intersects $R''''$ as $R''''$ is the pullback of some 
effective divisor on $V'$, a contradiction. Therefore, $(X'',B''+M'')$ is generalised plt.\\   
 
 \emph{Step 8.}
 In this step we finish the proof.
Replacing the given morphism $\phi\colon X'\to X$ we can assume $X'$ is a common resolution 
of $X,X''$. Recall
$$
K_{X'}+B'+M':=\phi^*(K_X+B+M)
$$ 
and let $F'$ be a general fibre of $X'\to Y'$. Since $X''\to Y'$ has relative dimension one and since 
the exceptional locus of $X'\to X''$ maps onto a closed subset of $X''$ of codimension $\ge 2$, 
we see that $X'\to X''$ is an isomorphism over the generic point of $Y'$. Thus if $F''$ is the fibre 
of $X''\to Y'$ corresponding to $F'$, then  
$(F',B'|_{F'})$ is isomorphic to $(F'',B''|_{F''})$ which is in turn isomorphic to 
$(\PP^1,p_1+p_2)$ for distinct points $p_1,p_2$. 
Moreover, if $S',T'$ on $X'$ are the birational transforms of $S'',T''$, then 
$\rddown{B'}=S'+T'$ as we showed that $S'',T''$ are the only generalised non-klt places of $(X'',B''+M'')$. 
In addition, $M'|_{F'}\equiv 0$ as we already showed that $M''|_{F''}\equiv 0$.

It is then clear that the images of $S',T'$ on $X$ are the only generalised non-klt centres of $(X,B+M)$. 
Since $\Nklt(X,B+M)$ is not connected, the images of $S',T'$ are disjoint and each gives a  
connected component of $\Nklt(X,B+M)$.   

\end{proof}

\begin{proof}(of Theorem \ref{t-anti-nef-pairs-2-generalised})
Cases (1) and (2) were proved in Lemmas \ref{l-anti-nef-pairs-2-generalised-(1)}, \ref{l-anti-nef-pairs-2-generalised-(2)}, respectively.  We treat case (3). Assume that the fibre of 
$$
\Nklt(X,B+M)\to Z
$$ 
over $z$ is not connected. Then by Theorem \ref{t-anti-nef-pairs-1-generalised}, after base change to an 
\'etale neighbourhood of $z$, we can assume that 
$(X,B+M)$ is generalised lc and that replacing $X'$ with a high resolution, there is a contraction 
$X'\to Y'/Z$ such that if 
$$
K_{X'}+B'+M':=\phi^*(K_X+B+M)
$$ 
and if $F'$ is a general fibre of $X'\to Y'$, 
then $M'|_{F'}\equiv 0$ and $(F',B'|_{F'})$ is isomorphic to $(\PP^1,p_1+p_2)$ for distinct points $p_1,p_2$.
Moreover, $\rddown{B'}$ has exactly two disjoint components $S',T'$, both horizontal over $Y'$.

But then   
$-(K_{X'}+B'+M')-tS'$ is not pseudo-effective over $Y'$ for any real number $t>0$ because 
$$
-(K_{X'}+B'+M')\cdot F'=-(K_{X'}+B')\cdot F'=0
$$ 
while $S'\cdot F'>0$. Therefore, $-(K_{X'}+B'+M')-tS'$ is not pseudo-effective over $Z$ for any $t>0$. 
This contradicts the assumption that 
$$
\tau_{S'}(-(K_X+B+M)/Z)>0.
$$

\end{proof}

\begin{proof}(of Theorem \ref{t-anti-nef-pairs-1})
This is a special case of Theorem \ref{t-anti-nef-pairs-1-generalised}.

\end{proof}

\begin{proof}(of Theorem \ref{t-anti-nef-pairs-2})
This is a special case of Theorem \ref{t-anti-nef-pairs-2-generalised}.

\end{proof}


\section{\bf Non-klt loci for mirror symmetry}

In this section we prove Theorem \ref{t-mirror-symmetry} and Corollary \ref{cor-mirror-symmetry}. 
First we prove a generalised 
version of Theorem \ref{t-mirror-symmetry} which occupies much of the section. 
Similar to the previous section, using generalised pairs is crucial 
for the proofs.

\begin{thm}\label{t-mirror-symmetry-generalised}
Assume that 
\begin{enumerate}
\item $(X,B+M)$ is a projective $\Q$-factorial generalised dlt pair with data $X'\to X$ and $M'$ 
where $B,M'$ are $\Q$-divisors,

\item $K_X+B+M$ is pseudo-effective, 

\item $x\in X$ is a zero-dimensional generalised non-klt centre of $(X,B+M)$,
 
\item $x$ is not contained in the restricted base locus ${\bf B}_-(K_X+B+M)$, 
 
\item if $\psi\colon Y\to X$ is the blowup at $x$ with exceptional divisor $E$, then  we have 
$\tau_E(K_X+B+M)=0$, that is, 
$$
\psi^*(K_X+B+M)-tE
$$ 
is not pseudo-effective for any real number $t>0$.
\end{enumerate}

Then $(X,B+M)$ has a good minimal model which is generalised log Calabi--Yau. More precisely, we can run a 
minimal model program on $K_X+B+M$ ending with a minimal model $(X'',B''+M'')$ with $K_{X''}+B''+M''\sim_\Q 0$.
\end{thm}

Note that since $(X,B+M)$ is generalised dlt and $x$ is a zero-dimensional non-klt centre, 
$x$ is a smooth point of $X$.

Before giving the proof of the theorem we prove several lemmas. We start with some basic properties of 
the pseudo-effective threshold $\tau$.

\subsection{Pseudo-effective thresholds}

\begin{lem}\label{l-tau-div-with-anti-ample-contraction}
Assume that 
\begin{itemize}
\item $X$ is a normal projective variety,

\item $L$ is a pseudo-effective $\R$-Cartier  $\R$-divisor on $X$, 

\item $\phi\colon Y\to X$ is a birational contraction from a normal variety and $S$ is a 
prime divisor on $Y$ such that $-S$ is ample over $X$, 

\item $T$ is a prime divisor over $X$ whose centre on $X$ is contained in the centre of $S$, and 

\item $\tau_S(L)>0$.  
\end{itemize}
Then
$
\tau_T(L)>0.
$
\end{lem}
\begin{proof}
Take a resolution $\alpha\colon V\to X$ on which $T$ is a divisor and such that the induced map 
$\beta\colon V\bir Y$ is a morphism.  
Assume $\tau_S(L)>0$. Then $\phi^*L-tS$ is pseudo-effective for some real number $t>0$. 
On the other hand, 
 $-S$ is ample over $X$, hence $\phi^{-1}\{x\}\subseteq S$ for every $x\in\phi(S)$, so $S=\phi^{-1}(\phi(S))$. 
Since the centre of $T$ on $X$ is contained in the centre of $S$, 
we see that the centre of $T$ on $Y$ is contained in $S$. Then 
the coefficient of $T$ in $\beta^*S$, say $e$, is positive. But then 
$$
\alpha^*L-teT=\beta^*(\phi^*L-tS)+\beta^*tS-teT
$$
is pseudo-effective as $\beta^*(\phi^*L-tS)$ is pseudo-effective and $t\beta^*S-teT\ge 0$. 
Therefore,  $\tau_T(L)>0$.
 
\end{proof}

\begin{lem}\label{l-tau-for-non-canonical-places-0}
Assume that 
\begin{itemize}
\item $(X,B+M)$ is a projective generalised lc pair with data $\phi\colon X'\to X$ and $M'$, 

\item $(X,C)$ is klt for some boundary $C$,  

\item $K_X+B+M$ is pseudo-effective, 

\item $S$ is a prime divisor over $X$ with 
$$
a(S,X,B+M)<1,
$$

\item $T$ is a prime divisor over $X$ whose centre on $X$ is contained in the centre of $S$, and 

\item
$
\tau_S(K_X+B+M)>0.
$
\end{itemize}
Then
$$
\tau_T(K_X+B+M)>0.
$$
\end{lem}
\begin{proof}
By Lemma \ref{l-tau-for-non-canonical-places-extraction}, there is a birational contraction 
$\psi\colon Y\to X$ from a normal variety such that $S$ is a divisor on $Y$ and $-S$ is ample over $X$. 
Thus by Lemma \ref{l-tau-div-with-anti-ample-contraction}, $\tau_S(K_X+B+M)>0$ implies $\tau_T(K_X+B+M)>0$.
 
\end{proof}

\begin{lem}\label{lem-tau-for-non-canonical-places}
Assume that 
\begin{itemize}
\item $(X,B+M)$ is a projective generalised lc pair with data $\phi\colon X'\to X$ and $M'$, 

\item $(X,C)$ is klt for some boundary $C$,  

\item $K_X+B+M$ is pseudo-effective, 

\item $S,T$ are prime divisors over $X$ with equal centre on $X$, and 

\item the generalised log discrepancies satisfy  
$$
a(S,X,B+M)< 1 ~~~\mbox{and}~~~a(T,X,B+M)< 1.
$$
\end{itemize}
Then 
$$
\tau_S(K_X+B+M)>0 ~~~~~~ \mbox{iff}~~~~~~\tau_T(K_X+B+M)>0.
$$
\end{lem}
\begin{proof}
This follows from Lemma \ref{l-tau-for-non-canonical-places-0}. 

\end{proof}

Next we look at pseudo-effective thresholds when we modify our pair birationally in certain situations.

\begin{lem}\label{l-tau-B-to-Delta}
Assume that 
\begin{itemize}
\item $(X,B+M)$ is a projective generalised lc pair with data $\phi\colon X'\to X$ and $M'$, 

\item $(X,C)$ is klt for some boundary $C$,  

\item $K_X+B+M$ is pseudo-effective, 

\item $S$ is a prime divisor over $X$ with 
$$
a(S,X,B+M)=0,
$$
and 

\item $\tau_S(K_X+B+M)=0$.
\end{itemize}
Writing 
$$
K_{X'}+B'+M'=\phi^*(K_X+B+M),
$$
let $\Delta'=B'+R'\ge 0$ where $R'\ge 0$ is exceptional over $X$, and assume $(X',\Delta'+M')$ is 
generalised lc.
Then
$$
\tau_S(K_{X'}+\Delta'+M')=0.
$$
\end{lem}
\begin{proof}
By Lemma \ref{l-tau-for-non-canonical-places-extraction}, 
there is a birational contraction $\psi\colon Y\to X$ from a normal variety such that $S$ 
is a divisor on $Y$ and $-S$ is ample over $X$. In particular, $S$ contains the exceptional locus of $\psi$, 
so $\psi$ does not contract any divisor except possibly $S$. 

Assume
$$
\tau_S(K_{X'}+\Delta'+M')>0.
$$
Let $\rho\colon W'\to X'$ be a resolution so that the induced map $\alpha\colon W'\bir Y$ is a morphism. 
Let $S'$ on $W'$ be the birational transform of $S$.
Then 
$$
\rho^*(K_{X'}+\Delta'+M')-tS'
$$ 
is pseudo-effective for some real number $t>0$. By assumption, $S'$ is a generalised non-klt place of $(X,B+M)$, 
so it is a generalised non-klt place of $(X',B'+M')$, hence 
$S'$ is not a component of $\rho^*R'$ otherwise $(X',\Delta'+M')$ would not be 
generalised lc. 

By assumption, $R'$ is exceptional over $X$, hence $\rho^*R'$ is exceptional over $X$.
Then since $S'$ is not a component of $\rho^*R'$ and since the only possible exceptional 
divisor of $Y\to X$ is $S$, we deduce that  
$\rho^*R'$ is exceptional over $Y$. Thus 
$$
\psi^*(K_X+B+M)-tS=\alpha_*(\rho^*(\phi^*(K_X+B+M))+\rho^*R'-tS')
$$
$$
=\alpha_*(\rho^*(K_{X'}+B'+M')+\rho^*(\Delta'-B')-tS')
$$
$$
=\alpha_*(\rho^*(K_{X'}+\Delta'+M')-tS')
$$ 
is pseudo-effective. Therefore, we get  
$$
\tau_{S}(K_{X}+B+M)>0,
$$
a contradiction.

\end{proof}

\subsection{Lifting zero-dimensional non-klt centres}

\begin{lem}\label{l-lifting-0dim-lcc}
Assume that $(X,B+M)$ is a generalised dlt pair with data $\phi\colon X'\to X$ and $M'$, 
and $x\in X$ is a zero-dimensional generalised non-klt centre of $(X,B+M)$. 
Write 
$$
K_{X'}+B'+M'=\phi^*(K_X+B+M)
$$ 
and assume $(X',B')$ is log smooth. Then  
$(X',B'+M')$ has a zero-dimensional generalised non-klt centre $x'$ mapping to $x$.
\end{lem}
\begin{proof}

Note that here we are considering $(X',B'+M')$ as a generalised sub-pair with data $X'\to X'$ and $M'$.
The generalised non-klt centres of $(X',B'+M')$ are just the 
non-klt centres of $(X',B')$. So we want to show that $(X',B')$ has a zero-dimensional 
non-klt centre $x'$ mapping to $x$.

Since $(X,B+M)$ is generalised dlt and $x\in X$ is a generalised non-klt centre, by definition, 
$(X,B)$ is log smooth near $x$ and $M'=\phi^*M$ over a neighbourhood of $x$. Shrinking 
$X$ we can assume $(X,B)$ is log smooth and that $M'=\phi^*M$. Writing 
$$
K_{X'}+C':=\phi^*(K_X+B),
$$
 we see that 
$C'=B'$ because $M'=\phi^*M$, hence $x$ is a non-klt centre of $(X,B)$.  
Thus in this situation $M'$ is not relevant, so removing it from now on we can assume $M'=0$. 

We can assume $d:=\dim X>1$ as the lemma is obvious in dimension one. 
We use induction on dimension. Since $(X,B)$ is log smooth and $x$ is a zero-dimensional
non-klt centre, $x$ is an intersection point of 
$d$ components of $\rddown{B}$ passing through $x$. Let $S$ be one such component 
and let $S'$ be its birational transform on $X'$. 
By adjunction, define $K_S+B_S=(K_X+B)|_S$ and $K_{S'}+B_{S'}=(K_{X'}+B')|_{S'}$. 
Since $(X',B')$ is log smooth and $S'$ is a component of ${B'}$ with coefficient $1$, $(S',B_{S'})$ is log smooth.
Denoting the induced morphism $S'\to S$ by $\pi$, we have 
$$
K_{S'}+B_{S'}=\pi^*(K_S+B_S).
$$ 
Moreover, $(S,B_S)$ is log smooth and $x$ is a non-klt centre of $(S,B_S)$. Thus by induction, there is a zero-dimensional 
non-klt centre $x'$ of $(S',B_{S'})$ mapping onto $x$. But then $x'$ is also a non-klt centre 
of $(X',B')$.     

\end{proof}

\subsection{Ample models for certain generalised pairs}
In this subsection we show that certain generalised pairs have ample models in the relative birational setting.

\begin{lem}\label{l-lc-model-bir}
Assume that
\begin{itemize}
\item $(X,B+M)$ is an lc generalised pair with data $X'\overset{\phi}\to X\overset{g}\to Y$ and $M'$, 
\item $B,M'$ are $\Q$-divisors and $g$ is a birational contraction, 
\item $(Y,B_Y+M_Y)$ is generalised lc with nef part $M'$ where $K_Y+B_Y+M_Y$ is the pushdown of $K_X+B+M$, and 
\item $(Y,C)$ is klt for some boundary $C$. 
\end{itemize}
Then $(X,B+M)$ has an ample model over $Y$, i.e. 
$$
\bigoplus_{m\ge 0}g_*\mathcal{O}_X(m(K_X+B+M))
$$
is a finitely generated $\mathcal{O}_Y$-module.
\end{lem}
\begin{proof}
We can assume $\phi$ is a log resolution of $(X,B)$ and that $X'\to Y$ is a log resolution of $(Y,B_Y)$. 
Write 
$$
K_{X'}+B'+M'=\phi^*(K_X+B+M)
$$
and let $\Delta'$ be obtained from $B'$ by increasing the coefficient of every exceptional$/X$ prime divisor 
to $1$. Since $(X',\Delta')$ is log smooth and log canonical, no divisor in $\Delta'$ with coefficient less than $1$ can contain any stratum of intersections of divisors of coefficient $1$.
Thus replacing $(X,B+M)$ with $(X',\Delta'+M')$ we can assume that 
 $(X,B+M)$ is $\Q$-factorial generalised dlt.

Run an MMP on $K_X+B+M$ over $Y$ with scaling of some ample divisor. We reach a model $V$ on which 
the pushdown $K_V+B_V+M_V$ is numerically a limit of movable$/Y$ $\R$-divisors (see \ref{ss-MMP} or 
[\ref{B-mmodel-II}, Step 2 in the proof of Theorem 1.5] for relevant details). 
Replacing $X$ with $V$ we can assume that $K_X+B+M$ is numerically a limit of movable$/Y$ $\R$-divisors.

Since $K_Y+B_Y+M_Y$ is $\R$-Cartier, we can write 
$$
K_X+B+M+R=g^*(K_Y+B_Y+M_Y)
$$
where $R$ is exceptional over $Y$. Then $-R$ is numerically a limit of movable$/Y$ $\R$-divisors, 
so for any exceptional$/Y$ prime divisor $S\subset X$, the divisor $-R|_S$ is pseudo-effective over $g(S)$. 
Therefore, by the general negativity lemma (cf. [\ref{B-lc-flips}, Lemma 3.3]), 
we have $R\ge 0$ (as in the proof of Lemma \ref{l-Q-fact-dlt-model}, to apply the negativity 
lemma we can first do a base change 
to an uncountable ground field if necessary). 

Let $E$ be the sum of the prime exceptional divisors of $g$ that are not components of 
$B+R$. We do induction on the number of components of $E$. If $E=0$, 
then $\Supp (B+R)$ contains all the exceptional divisors 
of $g$, so $X$ is of Fano type over $Y$ because $(Y,C)$ is klt for some $C$ (by the relative version of 
[\ref{B-compl}, 2.13(7)] applied over $Y$), hence $(X,B+M)$ has an ample 
model over $Y$. 

Now assume $E\neq 0$. By construction $(X,B+M)$ is $\Q$-factorial generalised dlt because this property is preseved by running MMP (see \ref{ss-gpp}(2)). So since $E$ has no common component with $B$, 
$(X,B+tE+M)$ is generalised dlt for some small number $t>0$ by definition of generalised dlt pairs.
Thus running an MMP on 
$$
K_X+B+tE+M
$$ 
over $Y$ with scaling of some ample divisor contracts $E$ 
because by the general negativity lemma the right hand side of the equation  
$$
K_X+B+tE+M\equiv_Y tE-R
$$
eventually becomes anti-effective. 

Let $W\to U$ be the step of the MMP 
where a component $E$ is contracted for the first time. 
Then $X\bir W$ is an isomorphism in codimension one 
because the MMP cannot contract any divisor other than components of $E$. Moreover, 
$$
K_W+B_W+M_W\sim_{\Q,U} -R_W \sim_{\Q,U} 0
$$
because $W\to U$ is extremal and because $-R_W$ is numerically a limit of movable$/Y$ 
$\R$-divisors and the divisor contracted by $W\to U$ is not a component of $R_W$.
 Therefore, if $K_U+B_U+M_U$ has an ample model over $Y$, then $K_W+B_W+M_W$ 
has an ample model over $Y$ which in turn implies that $K_X+B+M$ has an ample model over $Y$. 

Note that $(U,B_U+M_U)$ is $\Q$-factorial generalised dlt because $X\bir U$ is a partial MMP on 
$$
K_X+B+tE+M,
$$ 
so 
$$
(U,B_U+tE_U+M_U)
$$ 
is $\Q$-factorial generalised dlt. So we can replace $(X,B+M)$ 
with $(U,B_U+M_U)$. We are then done by 
induction as $E$ has one less component.

\end{proof}

\subsection{Fibrations}
In this subsection we prove a few results in order to treat Theorem \ref{t-mirror-symmetry-generalised} 
when the underlying space admits a suitable fibration.

\begin{lem}\label{l-fibration}
Let $(X,B+M)$ be as in Theorem \ref{t-mirror-symmetry-generalised}. Assume that 
\begin{itemize}
\item $f\colon X\to Z$ is a contraction where $\dim Z>0$, 

\item $K_X+B+M\sim_{\Q,Z} 0$, and 

\item $X$ is of Fano type over $Z$. 
\end{itemize}
Then there is a high resolution $Z'\to Z$ and a closed point 
$z'\in Z'$ such that 
\begin{itemize}
\item $z'$ maps to $f(x)$,

\item $z'$ is a generalised non-klt centre of $(Z',B_{Z'}+M_{Z'})$ where the latter is given by adjunction as in 
\ref{ss-Adjunction-gen-pair},

\item and we have 
$$
\tau_{G'}(K_{Z}+B_{Z}+M_{Z})=0
$$ 
where $G'$ is the exceptional divisor of the blowup 
of $Z'$ at $z'$.
\end{itemize}
\end{lem}
\begin{proof}
By adjunction, for each birational contraction $Z'\to Z$ where $Z'$ is normal, we get  
$(Z',B_{Z'}+M_{Z'})$ as defined in \ref{ss-Adjunction-gen-pair}. Assuming $Z'\to Z$ is a high log resolution 
of $(Z,B_Z)$, $M_{Z'}$ is nef. In particular, $(Z',B_{Z'}+M_{Z'})$ is a generalised sub-pair and $(Z,B_Z+M_Z)$ 
is a generalised pair with nef part $M_{Z'}$. 
Replace the given morphism $\phi\colon X'\to X$ with a high log resolution of $(X,B)$ 
so that the induced map $f'\colon X'\bir Z'$ is a morphism. 
Since $K_X+B+M$ is pseudo-effective and 
$$
K_X+B+M\sim_\Q f^*(K_Z+B_Z+M_Z),
$$ 
$K_Z+B_Z+M_Z$ is pseudo-effective.

Write 
$$
K_{X'}+B'+M'=\phi^*(K_X+B+M)
$$
where $M'$ is the nef part of $(X,B+M)$.
By assumption, $(X,B+M)$ is $\Q$-factorial generalised dlt and $x$ is a zero-dimensional generalised non-klt centre. Moreover, $(X',B')$ is log smooth as $\phi$ is assumed to be a log resolution of $(X,B)$.
Applying Lemma \ref{l-lifting-0dim-lcc}, there is a zero-dimensional generalised non-klt 
centre $x'$ of $(X',B'+M')$ mapping to $x$.
 
Let $z'=f'(x')$. We will show $z'$ is a generalised non-klt centre of $(Z',B_{Z'}+M_{Z'})$. Let $E'$ be the 
exceptional divisor of the blowup of $X'$ at $x'$. Since $x'$ is a generalised non-klt centre of $(X',B'+M')$, 
it is a non-klt centre of $(X',B')$, so 
there are $d:=\dim X$ components of $\rddown{B'}$ intersecting transversally at $x'$. 
So 
$$
a(E',X',B'+M')=0,
$$ 
that is, $E'$ is a non-klt place of $(X',B'+M')$. 

Now pick resolutions $Z''\to Z'$ and $X''\to X'$ such that $f''\colon X''\bir Z''$ is a 
morphism and if $E''\subset X''$ is the birational transform of $E'$, then $H'':=f''(E'')$ is 
a divisor on $Z''$. Write $K_{X''}+B''+M''$ for the pullback of $K_{X'}+B'+M'$. Then 
$E''$ is a component of $\rddown{B''}$, hence $H''$ is a component of $\rddown{B_{Z''}}$ 
because the generalised lc threshold of $f''^*H''$ with respect to $(X'',B''+M'')$ 
over the generic point of $H''$ is zero as $E''\le f''^*H''$. This shows that $z'$ is a generalised non-klt 
centre of $(Z',B_{Z'}+M_{Z'})$ as claimed since $H''$ maps to $z'$ as $E''$ maps to $x'$. 

Let $G'$ be the exceptional divisor of the blowup of $Z'$ at $z'$.
By assumption, $(Z',B_{Z'})$ is log smooth and $M_{Z'}$ is the nef part of $(Z',B_{Z'}+M_{Z'})$. 
Since $z'$ is a generalised non-klt centre of $(Z',B_{Z'}+M_{Z'})$, 
$$
a(G',Z,B_{Z}+M_{Z})=a(G',Z',B_{Z'}+M_{Z'})=0.
$$ 
On the other hand, by the previous paragraph, 
$$
a(H'',Z,B_{Z}+M_{Z})=a(H'',Z',B_{Z'}+M_{Z'})=0.
$$

Now since $X$ is of Fano type over $Z$, there is a boundary $D$ on $X$ such that $(X,D)$ is klt and 
$K_X+D\sim_{\Q,Z} 0$. Thus applying adjunction, we get $(Z,D_Z+N_Z)$ which is generalised klt. 
From this we get $C$ so that $(Z,C)$ is klt (as in the proof of Lemma \ref{l-tau-for-non-canonical-places-extraction}). 
Moreover, $G',H''$ both map to $f(x)$.
Thus, applying Lemma \ref{lem-tau-for-non-canonical-places} to $(Z,B_Z+M_Z), G',H''$, we see that to prove that  
$$
\tau_{G'}(K_{Z}+B_{Z}+M_{Z})=0
$$ 
it is enough to prove that   
$$
\tau_{H''}(K_{Z}+B_{Z}+M_{Z})=0.
$$  

By assumption, 
$$
\tau_{E}(K_{X'}+B'+M')=\tau_{E}(K_{X}+B+M)=0
$$ 
where $E$ is the exceptional divisor of the blowup of $X$ at $x$. Also 
$$
a(E,X,B+M)=0
$$
and 
$$
a(E',{X},B+M)=a(E',{X'},B'+M')=0
$$
where $E'$ is the exceptional divisor of the blowup of $X'$ at $x'$.
 Then by Lemma \ref{lem-tau-for-non-canonical-places},
$$
\tau_{E''}(K_{X}+B+M)=\tau_{E'}(K_{X}+B+M)=0
$$
where recall that $E''\subset X''$ is the birational transform of $E'$.
Thus
$$
K_{X''}+B''+M''-tE''
$$ 
is not pseudo-effective for any real number $t>0$.  
This in turn implies that  
$$
K_{X''}+B''+M''-tf''^*H''
$$ 
is not pseudo-effective for any real number $t>0$ 
because $E''\le f''^*H''$. But then 
$$
K_{Z''}+B_{Z''}+M_{Z''}-tH''
$$ 
is not pseudo-effective as 
$$
K_{X''}+B''+M''-tf''^*H''\sim_\Q f''^*(K_{Z''}+B_{Z''}+M_{Z''}-tH'').
$$
Therefore, 
$$
\psi^*(K_{Z}+B_{Z}+M_{Z})-tH''
$$
is not pseudo-effective for any $t>0$ where $\psi$ denotes $Z''\to Z$, hence 
$$
\tau_{H''}(K_{Z}+B_{Z}+M_{Z})=0
$$
as required.

\end{proof}

\begin{lem}\label{l-fibration-2}
Assume that 
\begin{itemize}
\item $(X,B+M)$ is a projective $\Q$-factorial generalised lc pair, 
\item $K_X+B+M$ is pseudo-effective,
\item $x\in X$ is a point not contained in ${\bf B}_-(K_X+B+M)$, 
\item $f\colon X\to Z$ is a contraction with $\dim Z>0$ and $K_X+B+M\sim_\R f^*L$ for some $\R$-Cartier $\R$-divisor $L$, and 
\item $X$ is of Fano type over $Z$.
\end{itemize}
Then 
$$
f(x)\notin {\bf B}_-(L).
$$ 
\end{lem}
\begin{proof}
Pick an ample Cartier divisor $A_Z$ on $Z$ and let $A=f^*A_Z$. 
We want to show that 
$$
z:=f(x)\notin {\bf B}(L+tA_Z)
$$ 
for any $t\in \R^{>0}$. Assume otherwise, that is, assume that 
$$
z\in {\bf B}(L+tA_Z)
$$ 
for some $t\in \R^{>0}$. Thus $z$ belongs to the support of every divisor 
$0\le R_Z\sim_\R L+tA_Z$. Then $f^{-1}\{z\}\subset \Supp R$ for every divisor 
$$
0\le R\sim_\R K_X+B+M+tA
$$
because any such $R$ is the pullback of an $R_Z$ as above. Therefore,  
$$
x\in {\bf B}(K_X+B+M+tA).
$$

Since $X$ is of Fano type over $Z$, $B+M$ is big over $Z$. Moreover, $B+M$ has a  
minimal model $Y$ over $Z$ on which $B_Y+M_Y$ is semi-ample over $Z$.  
Then  $s(B_Y+M_Y)+tA_Y$ is semi-ample for every $0<s\ll t$ (this can be seen by considering 
the ample model of $B_Y+M_Y$ over $Z$). Pick one such number $s$ and pick a general
$$
0\le D_Y\sim_\R s(B_Y+M_Y)+tA_Y. 
$$

Since $X$ is of Fano type over $Z$, $Y$ is of Fano type over $Z$, hence $Y$ has $\Q$-factorial klt singularities.
 Thus 
$$
(Y,(1-s)B_Y+(1-s)M_Y)
$$ 
has generalised klt singularities with nef part $(1-s)M'$ where replacing $X'$ we are assuming 
the induced map $X'\bir Y$ is a morphism. Since $D_Y$ is general, we deduce that 
$$
(Y,(1-s)B_Y+D_Y+(1-s)M_Y)
$$ 
is generalised klt with nef part $(1-s)M'$.

On the other hand, since $X\bir Y$ is an MMP on $B+M$ over $Z$, it is also an MMP on 
$s(B+M)+tA$. Thus $D_Y$ determines a unique divisor 
$$
0\le D\sim_\R s(B+M)+tA 
$$
whose pushdown to $Y$ is $D_Y$. Then the pair 
$$
(X,(1-s)B+D+(1-s)M)
$$ 
is generalised klt with nef part $(1-s)M'$ because  
$$
K_X+(1-s)B+D+(1-s)M\sim_{\R,Z} K_X+(1-s)B+sB+sM+tA+(1-s)M
$$
$$
=K_X+B+M+tA\sim_{\R,Z} 0
$$
which ensures that the generalised log discrepancies of 
$$
(X,(1-s)B+D+(1-s)M)
$$
coincide with those of 
$$
(Y,(1-s)B_Y+D_Y+(1-s)M_Y)
$$ 
which is generalised klt with nef part $(1-s)M'$. 
Moreover, $D$ is big.

Now, by [\ref{BZh}, Lemma 4.4], we can run an MMP on 
$$
K_X+(1-s)B+D+(1-s)M
$$ 
which ends with a good minimal model, say $V$. By the first paragraph, 
$$
x\in {\bf B}(K_X+B+M+tA)
$$
which means that $X\bir V$ is not an isomorphism near $x$. That is, $x$ belongs to 
the exceptional locus of some step of the MMP. But then 
$$
x\in {\bf B}(K_X+B+M+tA+uH)
$$
where $H$ is an ample divisor and $u>0$ is a sufficiently small real number. 
Therefore, 
$$
x\in {\bf B}(K_X+B+M+uH)
$$
as $A$ is semi-ample, hence 
$$
x\in {\bf B}_-(K_X+B+M),
$$
a contradiction.

\end{proof}

\begin{lem}\label{l-mirror-symmetry-generalised-induction}
Assume that Theorem \ref{t-mirror-symmetry-generalised} holds in dimension $\le d-1$. 
Suppose  
\begin{itemize}
\item $(X,B+M)$ is as in Theorem \ref{t-mirror-symmetry-generalised} in dimension $d$ with data 
$X'\to X$ and $M'$, 

\item $X\to Y$ is a birational contraction and $Y\to Z$ is a non-birational contraction, 

\item $(Y,B_Y+M_Y)$ is generalised lc with nef part $M'$, where $K_Y+B_Y+M_Y$ denotes the pushdown of $K_X+B+M$,

\item $Y$ is of Fano type over $Z$, and 

\item $K_Y+B_Y+M_Y\sim_{\Q,Z} 0$. 
\end{itemize}
Then 
$$
\kappa_\sigma(K_X+B+M)=\kappa(K_X+B+M)=0.
$$
\end{lem}
\begin{proof}
Note that $\kappa_\sigma$ is defined as in [\ref{Nakayama}, Definition 6.2.7].\\

\emph{Step 1.} 
In this step we consider an ample model of $(X,B+M)$ over $Y$ from which we derive a certain minimal 
model over $Z$. Since $Y$ is of Fano type 
over $Z$, $(Y,C)$ is klt and $K_Y+C\sim_{\Q,Z} 0$ for some $C$. Thus 
by Lemma \ref{l-lc-model-bir}, $(X,B+M)$ has an ample model over $Y$, say $U$.
Denoting the morphism $U\to Y$ by $\pi$, we can write 
$$
K_U+B_U+M_U+R_U=\pi^*(K_Y+B_Y+M_Y)
$$ 
for some $R_U$ exceptional over $Y$. Then $-R_U$ is ample over $Y$, so $R_U\ge 0$ by the negativity lemma.
Moreover, $\Supp R_U$ contains every exceptional divisor of $\pi$. Therefore, $U$ is of Fano type over $Z$ 
as $Y$ is of Fano type over $Z$, by the relative version of [\ref{B-compl}, 2.13(7)].
In particular, $(U,B_U+M_U)$ has a minimal model over $Z$, say $V$. 

Note that since $K_U+B_U+M_U$ is pseudo-effective and 
$$
K_U+B_U+M_U+R_U\sim_{\Q,Z} 0,
$$  
$R_U$ is vertical over $Z$ and 
$$
K_U+B_U+M_U\sim_\Q 0
$$ 
over the generic point of $Z$.\\

\emph{Step 2.}
In this step, after replacing $X'$, we define a boundary on $X'$ and find a generalised 
non-klt centre mapping to $x$ and study its properties. Recall that $x\in X$ is a zero-dimensional generalised non-klt centre of $(X,B + M)$ as in the statement of Theorem \ref{t-mirror-symmetry-generalised}.
Replacing $X'$ we can assume that it is a log resolution of $(X,B)$ and that the induced map $\rho\colon X'\bir V$ 
is a morphism and a log resolution of $(V,B_V)$. Write 
$$
K_{X'}+B'+M'=\phi^*(K_X+B+M).
$$
Let $\Delta'$ be the sum of the birational transform of $B_V$ and the reduced exceptional 
divisor of $\rho$. Then $R':=\Delta'-B'$ is exceptional over $V$ as $\rho_*\Delta'=B_V=\rho_*B'$. 
Moreover, $R'\ge 0$: indeed, for any prime divisor $D'$ exceptional over $V$, we have 
$$
\mu_{D'}(\Delta'-B')=1-\mu_{D'}B'\ge 0.
$$ 

By Lemma \ref{l-lifting-0dim-lcc}, there is a zero-dimensional 
generalised non-klt centre $x'$ of $(X',B'+M')$ mapping to the given point $x$. 
Then $x'$ is also a generalised non-klt centre of $(X',\Delta'+M')$, and $x'\notin \Supp R'$ as $(X',\Delta'+M')$ 
is generalised lc. On the other hand, we claim that
$$
x'\notin {\bf B}_-(K_{X'}+\Delta'+M').
$$ 
This follows from
$$
{\bf B}_-(K_{X'}+\Delta'+M')\subseteq {\bf B}_-(K_{X'}+B'+M')\cup \Supp R'
$$
$$
\subseteq \phi^{-1}{\bf B}_-(K_{X}+B+M)\cup \Supp R'
$$
and the assumption 
$$
x\notin {\bf B}_-(K_{X}+B+M)
$$ 
and the fact $x'\notin \Supp R'$.\\ 

\emph{Step 3.}
In this step we construct a minimal model of $({X'},\Delta'+M')$ over $V$ 
and study its properties.
By construction, 
$$
K_{X'}+\Delta'+M'=\rho^*(K_V+B_V+M_V)+P'
$$
where $P'\ge 0$ is exceptional over $V$. 
Thus running an MMP on $K_{X'}+\Delta'+M'$ over $V$ with scaling of some ample 
divisor contracts $P'$ and ends with a minimal model $W/V$ because $P'$ cannot be a 
limit of movable$/V$ divisors as it is exceptional over $V$. In fact, $(W,\Delta_W+M_W)$ 
is a $\Q$-factorial generalised dlt model of $(V,B_V+M_V)$ where $\Delta_W+M_W$ is the pushdown of $\Delta'+M'$. 
The map $X'\bir W$ is an isomorphism near $x'$ because $x'$ is not contained in ${\bf B}_-(K_{X'}+\Delta'+M')$. 

Let $E$ (resp. $E'$) be the exceptional divisor of 
the blowup of $X$ at $x$ (resp. of $X'$ at $x'$). Then 
$$
a(E,X,B+M)=0=a(E',X',B'+M')=a(E',X,B+M).
$$
This implies 
$$
a(E',V,B_V+M_V)=a(E',W,\Delta_W+M_W)=a(E',X',\Delta'+M')=0
$$ 
by the previous paragraph.

Now from the assumption 
$$
\tau_E(K_X+B+M)=0
$$
we deduce that 
$$
\tau_{E'}(K_X+B+M)=0,
$$
by Lemma \ref{lem-tau-for-non-canonical-places}. 
This in turn implies 
$$
\tau_{E'}(K_V+B_V+M_V)=0
$$
because by the construction of $V$, 
$$
\phi^*(K_X+B+M)-\rho^*(K_V+B_V+M_V)\ge 0.
$$
Thus 
$$
\tau_{E'}(K_{W}+\Delta_W+M_W)=0.
$$\

\emph{Step 4.}
In this step we replace $X,Y,Z$ so that we can assume that 
$K_X+B+M\sim_{\Q,Z} 0$ and that $X$ is of Fano type over $Z$.
By Step 1, $K_V+B_V+M_V$ is nef over $Z$ and $V$ is of Fano type over $Z$ as $U$ is of Fano 
type over $Z$. Then $K_V+B_V+M_V$ is semi-ample over $Z$ by the Fano type property (see \ref{ss-Fano}).
Let $V\to T/Z$ be the contraction defined by $K_V+B_V+M_V$. Since $K_V+B_V+M_V\sim_\Q 0$ 
over the generic point of $Z$ by Step 1, $T\to Z$ is birational but $V\to T$ is not birational as we assumed $Y\to Z$ is not birational.
By construction,
$$
\kappa_\sigma(K_X+B+M)=\kappa_\sigma(K_V+B_V+M_V)=\kappa_\sigma(K_W+\Delta_W+M_W)
$$
and 
$$
\kappa(K_X+B+M)=\kappa(K_V+B_V+M_V)=\kappa(K_W+\Delta_W+M_W).
$$
Moreover, $W$ is of Fano type over $Z$ as $V$ is of Fano type over $Z$ (by the relative 
version of [\ref{B-compl}, 2.13(7)]), so $W$ is also of 
Fano type over $T$.
Therefore, replacing $Z$ with $T$, replacing $(X,B+M)$ with $(W,\Delta_W+M_W)$, 
replacing $x$ with the image of $x'$ in $W$, and 
replacing $Y$ with $V$,  from now on we can assume that $K_X+B+M\sim_{\Q,Z} 0$ and that 
$X$ is of Fano type over $Z$. In particular, there is $Q$ such that $(X,Q)$ is klt and 
$K_X+Q\sim_{\Q,Z} 0$. From this we can get $Q_Z$ so that $(Z,Q_Z)$ is klt in case $\dim Z>0$ 
(for this we can 
either apply [\ref{Ambro-lc-trivial}, Theorem 0.2] or use adjunction to get a generalised klt strucutre on $Z$ from which 
we can derive a klt boundary).\\ 

\emph{Step 5.}
In this step we settle the case $\dim Z=0$, and prepare for induction in the case $\dim Z>0$.
If $\dim Z=0$, then $K_X+B+M\sim_\Q 0$ by the previous step, 
so the lemma holds in this case. From now on we assume 
that $\dim Z>0$. Denote $X\to Z$ by $f$ and let $z:=f(x)$. For each birational contraction 
$Z'\to Z$ from a normal variety, consider $(Z',B_{Z'}+M_{Z'})$ given by adjunction as in \ref{ss-Adjunction-gen-pair}. 
Then  by Lemma \ref{l-fibration}, there exist a high resolution $\psi\colon Z'\to Z$ and a closed point 
$z'\in Z'$ mapping to $z$ such that 
\begin{itemize}
\item $z'$ is a generalised non-klt centre of $(Z',B_{Z'}+M_{Z'})$, so 
$$
a(E',Z,B_Z+M_Z)=a(E',Z',B_{Z'}+M_{Z'})=0,
$$
\item and 
$$
\tau_{G'}(K_{Z}+B_{Z}+M_{Z})=0
$$ 
where $G'$ is the exceptional divisor of the blowup of $Z'$ at $z'$.
\end{itemize}\

\emph{Step 6.}
In this step we finish the proof.
Let $\Delta_{Z'}:=B_{Z'}^{\ge 0}$ and $P_{Z'}=\Delta_{Z'}-B_{Z'}$. Then $z'$ is a generalised 
non-klt centre of $(Z',\Delta_{Z'}+M_{Z'})$ and $z'\notin \Supp P_{Z'}$ because $(Z',B_{Z'})$ 
is log smooth so only $\dim Z$ exceptional divisors can meet at $z'$, and all of these have 
coefficient $1$ in $B_{Z'}$. 
Moreover, by Lemma \ref{l-tau-B-to-Delta} applied to $(Z,B_Z+M_Z)$ we have  
$$
\tau_{G'}(K_{Z'}+\Delta_{Z'}+M_{Z'})=0.
$$ 
Moreover, by Lemma \ref{l-fibration-2}, 
$$
z\notin {\bf B}_-(K_Z+B_Z+M_Z),
$$
hence 
$$
z'\notin {\bf B}_-(K_{Z'}+\Delta_{Z'}+M_{Z'})\subseteq \psi^{-1} {\bf B}_-(K_Z+B_Z+M_Z)\cup \Supp P_{Z'}.
$$
Thus $(Z',\Delta_{Z'}+M_{Z'})$ satisfies the properties listed in Theorem \ref{t-mirror-symmetry-generalised}. 
Since we are assuming the theorem in dimension $\le d-1$, $(Z',\Delta_{Z'}+M_{Z'})$ has a minimal model 
which is generalised log Calabi--Yau. Therefore, 
$$
\kappa_\sigma(K_{Z'}+\Delta_{Z'}+M_{Z'})=\kappa(K_{Z'}+\Delta_{Z'}+M_{Z'})=0,
$$
so we have 
$$
\kappa_\sigma(K_{Z}+B_{Z}+M_{Z})=\kappa(K_{Z}+B_{Z}+M_{Z})=0.
$$
These in turn imply that  
$$
\kappa_\sigma(K_X+B+M)=\kappa(K_X+B+M)=0
$$
as desired, by [\ref{Nakayama}, Proposition 6.2.8].

\end{proof}

\subsection{Proofs of Theorems \ref{t-mirror-symmetry} and \ref{t-mirror-symmetry-generalised}}

\begin{proof}(of Theorem \ref{t-mirror-symmetry-generalised})
We will apply induction on dimension, so assume the theorem holds in lower dimension. 
The case $\dim =1$ is easy to verify.\\

\emph{Step 1.} 
In this step we prove the theorem assuming that we have  
$$
(*) \ \ \ \kappa_\sigma(K_X+B+M)=\kappa(K_X+B+M)=0.
$$
From this we get
$$
K_X+B+M \equiv N_\sigma(K_X+B+M)\ge 0,
$$
by [\ref{Nakayama}, Proposition 6.2.8], where $N_\sigma$ is defined as in [\ref{Nakayama}, Definition 2.1.8]. 
Run an MMP on $K_X+B+M$ with scaling of some ample divisor $A$. 
We show that the MMP terminates (the case $M=0$ was established in [\ref{Gongyo}] where 
more details can be found). 
Let $\lambda_i$ be the 
scaling numbers and $X_i\bir X_{i+1}$ the steps of the MMP. Then $\lim \lambda_i=0$ 
by \ref{ss-MMP}. Moreover, 
$$
K_{X_i}+B_i+M_i+\lambda_iA_i
$$ 
is semi-ample, so any divisor not contracted by $X_1\bir X_i$ is not a component of 
$$
{\bf{B}}(K_X+B+M+\lambda_iA).
$$ 
Thus any prime divisor not contracted by the MMP is not a component of the restricted base locus 
 ${\bf{B}}_{-}(K_X+B+M)$. 
On the other hand, by definition of $N_\sigma$, 
$$
\Supp N_\sigma(K_X+B+M)
$$ 
is contained in ${\bf{B}}_{-}(K_X+B+M)$. Therefore, $N_\sigma(K_X+B+M)$ is contracted by the MMP. 
This ensures  
$$
K_{X_i}+B_i+M_i\equiv 0
$$ 
for $i\gg 0$. 
Since 
$$
\kappa(K_X+B+M)=0,
$$
we deduce that 
$$
K_{X_i}+B_i+M_i\sim_\Q 0,
$$ 
for $i\gg 0$. Thus the MMP ends with a good minimal model as required.
It is then enough to show that $(X,B+M)$ satisfies $(*)$.\\

\emph{Step 2.}  
In this step we modify $(X,B+M)$ so that we can assume 
$$
\tau_S(K_X+B+M)=0
$$ 
for some component $S$ of $\rddown{B}$.
Let $Y\to X$ be the blowup of $X$ at $x$ with exceptional divisor $E$. 
We can assume that the induced map $X'\bir Y$ is a morphism and that $\phi\colon X'\to X$ 
is a log resolution of $(X,B)$. Writing 
$$
K_{X'}+B'+M'=\phi^*(K_X+B+M),
$$
let $K_Y+B_Y+M_Y$ be the pushdown of $K_{X'}+B'+M'$.
We consider $(Y,B_Y+M_Y)$ as a generalised pair with data $X'\to Y$ and $M'$. 
There is a zero-dimensional generalised non-klt centre $y$ of $(Y,B_Y+M_Y)$ mapping to $x$ 
because there is a zero-dimensional generalised non-klt centre $x'$ of 
$(X',B'+M')$ mapping to $x$, by Lemma \ref{l-lifting-0dim-lcc}, and we can take $y$ to be the 
image of $x'$ on $Y$; alternatively 
we can find $y$ using the fact that $Y\to X$ is the blowup of a smooth point $x$ 
and that $(X,B)$ is log smooth near $x$; this in particular shows that 
$(Y,B_Y+M_Y)$ is $\Q$-factorial generalised dlt, so $(Y,B_Y)$ 
is log smooth near $y$. 

Let $G$ be the exceptional divisor of the blowup of $Y$ at $y$. Then 
$$
a(E,X,B+M)=0=a(G,Y,B_Y+M_Y)=a(G,X,B+M).
$$
Also by assumption,
$$
\tau_{E}(K_Y+B_Y+M_Y)=\tau_{E}(K_X+B+M)=0.
$$
Applying Lemma \ref{lem-tau-for-non-canonical-places}, we see that 
$$
\tau_{G}(K_{Y}+B_Y+M_Y)=\tau_{G}(K_{X}+B+M)=0.
$$

Now $(Y,B_Y+M_Y),y$ satisfies the properties (1)-(3), (5) listed in Theorem \ref{t-mirror-symmetry-generalised}. 
It also satisfies (4), that is,
$$
y\notin {\bf B}_-(K_{Y}+B_Y+M_Y)
$$ 
as 
$$
{\bf B}_-(K_{Y}+B_Y+M_Y)\subseteq \alpha^{-1} {\bf B}_-(K_{X}+B+M)
$$
where $\alpha$ denotes $Y\to X$.
Therefore, replacing $(X,B+M),x$ with $(Y,B_Y+M_Y),y$, we can assume that 
$$
\tau_S(K_X+B+M)=0
$$ 
for some component $S$ of $\rddown{B}$.\\

\emph{Step 3.}
In this step we choose a number $t$ and find a Mori fibre space for $({X},B+M-t\rddown{B})$.
Let $\Phi$ be the set of the coefficients of $B$ and $p$ be a natural number so that 
$pM'$ is Cartier. 
Pick a small rational number $t>0$ as in Lemma \ref{l-application-of-local-global-ACC} 
for the data $\dim X,p,\Phi$. By the previous step,  
$$
K_{X}+B+M-t\rddown{B}
$$ 
is not pseudo-effective. Thus we can run an MMP on 
$$
K_{X}+B+M-t\rddown{B}
$$ 
ending with a Mori fibre space $V\to Z$. 
Since $K_X+B+M$ is pseudo-effective, $K_V+B_V+M_V$ is nef over $Z$.

Writing 
$$
B_V-t\rddown{B_V}+M_V=B_V-\rddown{B_V}+(1-t)\rddown{B_V}+M_V,
$$ 
and applying Lemma \ref{l-application-of-local-global-ACC}, we deduce that 
$(V,B_V+M_V)$ is generalised lc and 
$$
K_V+B_V+M_V\equiv_Z 0.
$$\ 

\emph{Step 4.}
In this step we introduce a boundary $\Delta'$ on $X'$ and study its properties.
Replacing $X'$ we can assume that $\phi$ is a log resolution of $(X,B)$ and that 
the induced map $X'\bir V$ is a morphism. Recall 
$$
K_{X'}+B'+M'=\phi^*(K_X+B+M)
$$ 
from Step 2. Let $\Delta':=B'^{\ge 0}$. Applying Lemma \ref{l-lifting-0dim-lcc}, 
we see that there is a zero-dimensional generalised non-klt centre 
$x'$ of $(X',B'+M')$ mapping to $x$. Note that $x'$ is also a generalised non-klt centre 
of $(X',\Delta'+M')$ and $x'\notin \Supp (\Delta'-B')$. 
Moreover, if $G'$ is the exceptional divisor of the blowup of $X'$ at $x'$, then 
by Lemma \ref{lem-tau-for-non-canonical-places}, 
$$
\tau_{G'}(K_{X}+B+M)=0,
$$
so 
$$
\tau_{G'}(K_{X'}+\Delta'+M')=0
$$
by Lemma \ref{l-tau-B-to-Delta}. 

In addition, 
$$
x'\notin {\bf B}_-(K_{X'}+\Delta'+M')
$$ 
because 
$$
{\bf B}_-(K_{X'}+\Delta'+M')\subseteq \phi^{-1}{\bf B}_-(K_{X}+B+M)\cup \Supp (\Delta'-B')
$$
and $x\notin {\bf B}_-(K_{X}+B+M)$ and $x'\notin \Supp (\Delta'-B')$.\\

\emph{Step 5.}
In this step we finish the proof. 
By the previous step, we can replace $(X,B+M),x$ with $(X',\Delta'+M'),x'$, hence assume that there is a 
birational contraction $X\to V$ and a non-birational contraction $V\to Z$ such that 
\begin{itemize}
\item $(V,B_V+M_V)$ is generalised lc with nef part $M'$, where $K_V+B_V+M_V$ denotes the pushdown of $K_X+B+M$,

\item $V$ is of Fano type over $Z$, and 

\item $K_V+B_V+M_V\sim_{\Q,Z} 0$. 
\end{itemize}
Now since we are assuming Theorem \ref{t-mirror-symmetry-generalised} in lower 
dimension, applying Lemma \ref{l-mirror-symmetry-generalised-induction}, we get 
$$
 \kappa_\sigma(K_X+B+M)=\kappa(K_X+B+M)=0
$$
as desired.

\end{proof}

\begin{proof}(of Theorem \ref{t-mirror-symmetry})
This is a special case of Theorem \ref{t-mirror-symmetry-generalised}.

\end{proof}

\subsection{Proofs of Corollaries \ref{cor-mirror-symmetry-connectedness} and \ref{cor-mirror-symmetry}}

Before giving the proofs of the corollaries we make a bit of preparation.

\begin{lem}\label{l-cy-connected}
Let $(X,B)$ be a projective  dlt pair with $K_X+B\equiv 0$ having a zero-dimensional 
non-klt centre $x$. Assume that either $V=X$ or $V$ is a non-klt centre of $(X,B)$, and that  
$\dim V\ge 2$. Let $K_V+B_V=(K_X+B)|_V$ be given by adjunction ($B_V=B$ when $V=X$). Then the non-klt locus 
$\Nklt(V,B_V)=\rddown{B_V}$ is connected.
\end{lem}
\begin{proof}
The equality $\Nklt(V,B_V)=\rddown{B_V}$ follows from the assumption that $(X,B)$ is dlt 
which implies that $(V,B_V)$ is dlt.
Assume  that $\Nklt(V,B_V)$ is not connected.
By [\ref{Kollar-singMMP}, Theorem 4.40], all the minimal non-klt centres of $(X,B)$ are birational 
to each other, in particular, they have the same dimension. Since we already have  
a zero-dimensional non-klt centre $x$, all the minimal non-klt centres are zero-dimensional 
(we recall the proof of this fact below). 
Therefore, any minimal non-klt centre of $(V,B_V)$ is also zero-dimensional because any minimal 
non-klt centre of $(V,B_V)$ is a minimal non-klt centre of $(X,B)$.

Since we assumed that $\Nklt(V,B_V)=\rddown{B_V}$  is not connected, by Theorem \ref{t-anti-nef-pairs-1}, 
$(V,B_V)$ has exactly two disjoint non-klt centres. As $(V,B_V)$ is dlt, these centres 
are two disjoint irreducible components of $\rddown{B_V}$. This contradicts the fact that 
$(V,B_V)$ has a zero-dimensional non-klt centre. 

\end{proof}

The lemma does not hold if we relax the dlt assumption to lc; see the example 
given after Theorem \ref{t-anti-nef-pairs-1}.

In the proof of the lemma we used the fact that all the minimal non-klt centres of $(X,B)$ 
are zero-dimensional. For convenience, we give the proof here essentially following 
[\ref{Kollar-singMMP}, Theorem 4.40]. First applying Theorem \ref{t-anti-nef-pairs-1}, 
we see that $\rddown{B}$ is connected because $(X,B)$ has a zero-dimensional non-klt centre $x$. 
Now pick any minimal non-klt centre $W$ of 
$(X,B)$. Then there exist components $D_1,\dots,D_r$ of $\rddown{B}$ such that 
$x\in D_1$, $W\subset D_r$, $D_i$ intersects $D_{i+1}$ for each $1\le i<r$, and that $r$ 
is minimal with these properties. Then by induction on dimension, any minimal non-klt centre of 
$(D_1,B_{D_1})$ is zero-dimensional where $K_{D_i}+B_{D_i}=(K_X+B)|_{D_i}$. 
In particular, $D_1\cap D_2$ contains a zero-dimensional non-klt centre $x_2$ of $(D_1,B_{D_1})$
which is in turn a non-klt centre of $(X,B)$. Repeating this argument we find a zero-dimensional 
non-klt centre of $(X,B)$ in each $D_i$, in particular, in $D_r$. 
But $W\subset D_r$, so applying induction to $(D_r,B_{D_r})$ implies that 
$\dim W=0$. 

\

\begin{proof}(of Corollary \ref{cor-mirror-symmetry-connectedness})
By Theorem \ref{t-mirror-symmetry}, we can run an MMP on $K_X+B$ ending with a minimal model 
$X'$ with $K_{X'}+B'\sim_\Q 0$. By assumption, no non-klt centre of $(X,B)$ is contained in ${\bf B}_-(K_{X}+B)$.
Thus $X\bir X'$ is an isomorphism near the generic point 
of each non-klt centre. So each non-klt centre has a birational transform on $X'$. 
On the other hand, if $W'$ is any non-klt centre of $(X',B')$, then $X\bir X'$ is an isomorphism 
over the generic point of $W'$ because $X\bir X'$ is an MMP on $K_X+B$.  
Thus there is a 1-1 correspondence between the non-klt centres of $(X,B)$ and $(X',B')$ 
given by birational transform. In particular, the image $x'$ on $X'$ of the given non-klt centre 
$x$ is a zero-dimensional non-klt centre of $(X',B')$. 

Assume that either $V=X$ or that $V$ is a non-klt centre of $(X,B)$.  
Define $K_V+B_V=(K_X+B)|_V$ by adjunction ($B_V=B$ if $V=X$). 
Since $(X,B)$ is dlt, $(V,B_V)$ is dlt, so 
$$
\Nklt(V,B_V)=\rddown{B_V}.
$$ 
Similarly, define $K_{V'}+B_{V'}=(K_{X'}+B')|_{V'}$ by adjunction.
Then 
$$
\Nklt(V',B_{V'})=\rddown{B_{V'}}.
$$ 
Each non-klt centre of $(V,B_V)$ is a non-klt centre of $(X,B)$, and similarly, 
each non-klt centre of $(V',B_{V'})$ is a non-klt centre of $(X',B')$.
Thus by the previous paragraph, $\rddown{B_{V'}}$ is the birational transform of $\rddown{B_{V}}$.

Now assume that $\dim V\ge 2$. Since $(X',B')$ is dlt with $K_{X'}+B'\sim_\Q 0$ having a 
zero-dimensional non-klt centre $x'$, applying Lemma \ref{l-cy-connected} shows that 
$\Nklt(V',B_{V'})=\rddown{B_{V'}}$ is connected.

Assume that $\Nklt(V,B_V)=\rddown{B_V}$ is not connected. 
We will derive a contradiction. 
Since $\rddown{B_V}$ is not connected but $\rddown{B_{V'}}$ is connected, 
we can find disjoint components $S,T$ of $\rddown{B_V}$ such that their birational 
transforms $S',T'$ on $X'$ intersect. Let $W'$ be an irreducible component of $S'\cap T'$. 
Then $S',T'$ are components of $\rddown{B_{V'}}$, hence 
$W'$ is a non-klt centre of $(V',B_{V'})$, so $W'$ is also a non-klt centre of $(X',B')$.
Thus $X\bir X'$ is an isomorphism over the generic point of $W'$, hence $W'$ is the birational 
transform of a non-klt centre $W$ of $(X,B)$, and $S,T$ both contain $W$, a contradiction. Therefore, 
$\rddown{B_{V}}$ is connected.

\end{proof}

\begin{proof}(of Corollary \ref{cor-mirror-symmetry})
For each $i$, let $0\le u_i\le a_i$ be a rational number such that $u_i\le 1$.
Let $\Delta=B-\sum u_iB_i$.
Then 
$$
K_X+\Delta=K_X+B-\sum u_iB_i\equiv \sum (a_i-u_i)B_i\ge 0.
$$
In particular, $K_X+\Delta$ is pseudo-effective and 
$$
{\bf B}_-(K_X+\Delta)\subseteq \Supp \sum (a_i-u_i)B_i\subseteq B-C
$$
 where $C$ is the sum of the good components of $B$.
On the other hand, 
by assumption, $C$ and $B-C$ have no common components, 
and $x$ is a zero-dimensional non-klt centre of $(X,C)$. Thus $x$ is not contained in $B-C$ which implies that 
$x\notin {\bf B}_-(K_X+\Delta)$.

Moreover, ${\Delta}\ge C$, so $x$ is a zero-dimensional non-klt centre of $(X,\Delta)$. 
Also 
$$
0\le \tau_E(K_X+\Delta)\le \tau_E(K_X+B)=0
$$
where $E$ is the exceptional divisor of the blow up of $X$ at $x$.
Therefore, $(X,\Delta)$ satisfies all the assumptions of Theorem \ref{t-mirror-symmetry}, so  
we can run an MMP on $K_X+\Delta$ which ends with a log minimal model $(X',\Delta')$ with $K_{X'}+\Delta'\sim_\Q 0$. 
In particular, taking $u_i=0$ for every $i$, we have $\Delta=B$, so we get the first claim of the 
corollary.  
   
Assume that $V$ is a stratum of $(X,C)$, that is, 
either $V=X$ or that $V$ is a non-klt centre of $(X,C)$. Assume $\dim V \ge 2$. We want to show that 
$C_V$ is connected where $K_V+C_V=(K_X+C)|_V$. 

From now on we assume that $u_i>0$ when $a_i>0$. Then $\rddown{\Delta}=C$, 
so the non-klt centres of $(X,\Delta)$ and $(X,C)$ are the same. 
So $C_V=\rddown{\Delta_V}$ where  $K_V+\Delta_V=(K_X+\Delta)|_V$. 
Moreover, no non-klt centre of $(X,C)$ is contained in $B-C$, so no non-klt centre of $(X,\Delta)$ 
is contained in $B-C$, hence 
by the first paragraph of this proof, no non-klt centre of $(X,\Delta)$ 
is contained in ${\bf B}_-(K_X+\Delta)$. Therefore, by Corollary \ref{cor-mirror-symmetry-connectedness}, 
$C_V=\rddown{\Delta_V}$ is connected as desired.

\end{proof}


\vspace{2cm}

\small
\textsc{Yau Mathematical Sciences Center} \endgraf
\textsc{JingZhai Building, Tsinghua University} \endgraf
\textsc{ Hai Dian District, Beijing, China 100084  } \endgraf
\vspace{0.5cm}
\email{Email: birkar@tsinghua.edu.cn\\}

\end{document}